\documentclass[a4paper,11pt,reqno]{amsart}
\usepackage{amsmath,amssymb,amsthm} 
\usepackage{graphics} 

\usepackage{epsfig}
\usepackage[colorlinks = true]{hyperref}
\usepackage{color}

\numberwithin{equation}{section}

\newtheorem{theorem}{Theorem}[section]
\newtheorem{lemma}[theorem]{Lemma}
\newtheorem{proposition}[theorem]{Proposition}

\newtheorem{corollary}[theorem]{Corollary}

\theoremstyle{definition}
\newtheorem{definition}[theorem]{Definition}
\newtheorem{remark}[theorem]{Remark}

\newcommand{\restr}{\mathop{\raisebox{-.127ex}{\reflectbox{\rotatebox[origin=br]{-90}{$\lnot$}}}}}

\newcommand{\R}{\mathbb{R}}
\newcommand{\N}{\mathbb{N}}
\newcommand{\C}{\mathbf{C}}

\newcommand{\eps}{\varepsilon}

\newcommand{\be}{\begin{equation}}
\newcommand{\ee}{\end{equation}}

\newcommand\lt{\left}
\newcommand\rt{\right}

\def\les{\lesssim}
\def\ges{\gtrsim}

\def \C{\mathbf{C}}

\def\EE{\mathbb{E}}
\def\PP{\mathbb{P}}
\def\diam{\operatorname{diam}}

\newcommand{\bra}[1]{\left( #1 \right)}
\newcommand{\sqa}[1]{\left[ #1 \right]}
\newcommand{\cur}[1]{\left\{ #1 \right\}}

\newcommand{\abs}[1]{\left| #1 \right|}
\newcommand{\nor}[1]{\left\| #1 \right\|}

\usepackage{enumerate}
\newcommand{\Lip}{\operatorname{Lip}}
\newcommand{\Wb}{Wb}
\newcommand{\bTSP}{b\operatorname{TSP}}
\newcommand{\TSP}{\operatorname{TSP}}
\renewcommand{\C}{\mathcal{C}}

\newcommand{\meas}{\mathfrak{m}}

\title{On the quadratic random matching problem in two-dimensional domains}
\thanks{L.A.\ thanks the support of the PRIN 2017 project ``Gradient Flows, Optimal Transport and Metric Measure Structures''.
 M.G.\ was partially supported by the  project 
ANR-18-CE40-0013 SHAPO financed by the French Agence Nationale de la Recherche (ANR).   D.T.\ was partially supported by Gnampa project 2020 ``Problemi di ottimizzazione con vincoli via trasporto ottimo e incertezza''.}

\author[L. Ambrosio]{Luigi Ambrosio}
\address{L.A.: Scuola Normale Superiore, 56126 Pisa, Italy}
\email{luigi.ambrosio@sns.it}
\author[M. Goldman]{Michael Goldman}
\address{M.G.: Universit\'e de Paris, CNRS, Sorbonne-Universit\'e,  Laboratoire Jacques-Louis Lions (LJLL), F-75005 Paris, 
France}
\email{goldman@u-paris.fr}
\author[D. Trevisan]{Dario Trevisan}
\address{D.T.: Dipartimento di Matematica, Università degli Studi di Pisa, 56125 Pisa, Italy  }
\email{dario.trevisan@unipi.it}
\date{}

\begin{document}

\maketitle

\begin{abstract}
We investigate the average minimum cost of a bipartite matching, with respect to the squared Euclidean distance,  between two samples of $n$ i.i.d.\ random points on a bounded Lipschitz domain in the Euclidean plane, whose common law is absolutely continuous with strictly positive H\"older continuous density. We confirm in particular the validity of a conjecture by D.~Benedetto and E.~Caglioti stating that the asymptotic cost as $n$ grows is given by the logarithm of $n$ multiplied by an explicit constant times the volume of the domain. Our proof relies on a reduction to the optimal transport problem  between the associated empirical measures and a Whitney-type decomposition of the domain, together with suitable upper and lower bounds for local and global contributions, both ultimately based on PDE tools. We further show how to extend our results to more general settings, including Riemannian manifolds, and also give an application to the asymptotic cost of the random quadratic bipartite travelling salesperson problem.
\end{abstract}

\section{Introduction}

The minimum weight perfect matching problem on bipartite graphs, also called assignment problem, is a combinatorial optimization problem which has been historically subject of intense research by several communities, well beyond operation research and algorithm theory, including combinatorics and graph theory \cite{lovasz2009matching}, probability and statistics \cite{Ta14} and even theoretical physics \cite{mezard1987spin, mezard2009information}. Applications related to planning and allocation of resources are classical, but recently they have seen an increased interest, e.g.\ in the online version of the problem, related to Internet advertising \cite{41870}. The assignment problem and its common linear programming relaxation, the optimal transport problem, provide also useful tools in machine learning and data science \cite{arjovsky2017wasserstein, peyre2019computational}, mostly because of its efficiency and versatility at discriminating between empirical distributions.  

In this article, we focus on the Euclidean formulation of the problem, where we are given two finite families of points $\bra{x_i}_{i=1}^n$, $\bra{y_j}_{j=1}^m \subseteq \R^d$, with $m \ge n$, and  the matching cost is defined, for a given parameter $p>0$,
\begin{equation}\label{eq:assignment} \min_{\sigma \in \mathcal{S}_{n,m}} \sum_{i=1}^n | x_i - y_{\sigma(i)} |^p,\end{equation}
where $\mathcal{S}_{n,m}$ denote the set of injective maps $\sigma: \cur{1,\ldots, n} \to \cur{1, \ldots, m}$. This corresponds to a minimum weight perfect matching problem
on the complete bipartite graph  and edge weights $w_{ij} = |x_i-y_j|^p$. In particular, when $n=m$, $\mathcal{S}_{n,n} = \mathcal{S}_{n}$ is the symmetric group of permutations over $n$ elements. 

A  linear programming approximation of \eqref{eq:assignment} is given by the Monge-Kantorovich optimal transport problem, where $\sigma$ is replaced by a transport plan or coupling, i.e., a matrix $(\pi_{ij})_{i=1,\ldots,n}^{j=1,\ldots,m} \in [0,1]^{n\times m}$ with $\sum_{i} \pi_{ij} = 1/m$, $\sum_{j} \pi_{ij} = 1/n$. Minimization among all couplings yields a particular instance of the Wasserstein cost of order $p>0$ between the empirical measures  $\frac 1 n \sum_{i=1}^n \delta_{x_i}$ and $\frac 1 m \sum_{j=1}^m \delta_{y_j}$,
\begin{equation}\label{eq:wasserstein-intro}  
\min_{\pi} \sum_{i=1}^{n}\sum_{j=1}^m \pi_{ij} |x_i - y_j|^p = W_p^p\bra{ \frac 1 n \sum_{i=1}^n \delta_{x_i}, \frac 1 m \sum_{j=1}^m \delta_{y_j}}.
\end{equation}
The Birkhoff-von Neumann  theorem \cite{lovasz2009matching} provides an exact  correspondence between  \eqref{eq:assignment} and \eqref{eq:wasserstein-intro} when $n=m$, i.e.,  both the costs and optimizers are the same, the optimal plan is as a permutation matrix (both up to a factor $1/n$). As a consequence, one can exploit the rich analytic structure of optimal transport and  use tools from  convex analysis and partial differential equations.

We further assume that the points $x_i = X_i$, $y_j = Y_j$ are samples of independent identically distributed (i.i.d.)~random variables with a common law. This is the paradigm of many problems in geometric probability, in particular related to the theory of Euclidean additive functionals \cite{steele1997probability, Yu98}. The random Euclidean bipartite matching problem is in fact known to be quite challenging: general results and techniques that give precise results in similar problems (e.g.\ the non-bipartite matching problem) fail here, most notably when the dimension of the space is small. 

The ``critical'' case turns out to be $d=2$, which we also consider in this work. In this case, Ajtai, Koml\'os and Tusn\'ady \cite{AKT84} first showed that for i.i.d.\ uniform samples on the square $(0,1)^2$ the asymptotics in the case $n=m \to \infty$ reads, for\footnote{
We use  the notation $A\sim B$ if both $A\les B$ and $B\les A$ hold, where in turn $A\les B$ means that there exists a global constant $C = C(d,p,\Omega)>0$ depending only on the dimension $d$, $p\ge 1$ and on the fixed domain $\Omega$ such that $A\le C B$.}  $p \ge 1$,
\begin{equation}\label{eq:AKT} \EE\sqa{ W_p^p\bra{ \frac 1 n \sum_{i=1}^n \delta_{X_i}, \frac 1 n \sum_{i=1}^n \delta_{Y_i}}} \sim
 \bra{ \frac{\log n}{n}}^{p/2}. \end{equation}
Talagrand \cite{talagrand1992ajtai} subsequently investigated the case of general laws, possibly not even absolutely continuous, providing a universal upper bound for $p=1$, with the same rate.

The parameter $p>0$ plays a relevant role, since the solution to both problems \eqref{eq:assignment} and \eqref{eq:wasserstein-intro} in general depend on $p$, and for $d=2$, $p<1$ the correct rate in \eqref{eq:AKT} turns out to be $n^{-p/2}$, i.e. no  logarithmic corrections appear \cite{BaBo}. The classical literature focused mostly on the case $p=1$ -- but also on the case $p=\infty$, see \cite{leighton1989tight, trillos2015rate}.

A recent breakthrough in  the  quadratic case $p=2$ was obtained by the statistical physics community, starting from the seminal work \cite{CaLuPaSi14} and further developed in \cite{CaSi15, sicuro_euclidean_2017}. By formally linearising the Monge-Amp\`ere equation around the constant density to obtain the Poisson equation, they argued in particular that the optimal transport cost should be  well-approximated by a negative Sobolev norm of the difference of the empirical measures. After a renormalization procedure to cut-off divergences, this led for uniform points on $(0,1)^2$ to the conjecture
\begin{equation}\label{eq:caracciolo} \lim_{ n\to \infty}  \frac{n}{\log n} \EE\sqa{ W_2^2\bra{ \frac 1 n \sum_{i=1}^n \delta_{X_i}, \frac 1 n \sum_{i=1}^n \delta_{Y_i}}} = \frac{1}{2 \pi}.\end{equation}
 A rigorous mathematical proof of \eqref{eq:caracciolo} was later obtained in \cite{AST}. Further  simplifications and improvements of this method have been done in \cite{AmGlau},  leading to quantitative bounds for optimizers \cite{ambrosio2019optimal}. See also \cite{ goldman2021fluctuation}  for a justification of the linearisation ansatz of \cite{CaLuPaSi14} down to mesoscopic scales  based on a large-scale regularity theory for the Monge-Amp\`ere equation \cite{GHO}. 
 Let us point out that for $p \ge 1$, with the exception of $p = 2$, even in the case of uniform points on $(0,1)^2$ it is still an open problem \cite[Section 4.3.3]{Ta14} to determine if the limit in \eqref{eq:AKT} exists, i.e. prove the existence of 
\[ \lim_{ n\to \infty} \bra{ \frac{n}{\log n}}^{p/2} \EE\sqa{ W_p^p\bra{ \frac 1 n \sum_{i=1}^n \delta_{X_i}, \frac 1 n \sum_{i=1}^n \delta_{Y_i}}}.\]

A a natural question is how \eqref{eq:caracciolo} should be modified when the uniform law is replaced by a different density. As a first step in this direction, it is proven in \cite{AmGlau,AST}  that \eqref{eq:caracciolo} holds on general closed compact Riemannian manifolds $(M, g)$, when the cost is the square of  the Riemannian distance and the law of the samples is the normalized Riemannian volume measure. This however does not even cover  the case of a uniform measure on a general convex set $M = \Omega \subseteq \R^2$, because of the presence of the boundary. For non-convex sets, an additional problem is that the Riemannian distance is different from the Euclidean one. On the other side, it was recently conjectured in  \cite[Conjecture 2]{BeCa} that for every bounded connected open set  $\Omega\subset \R^2$ with  smooth boundary and every  smooth, uniformly positive and bounded density with respect to the  Lebesgue, the limit \eqref{eq:caracciolo} holds true with the right-hand side multiplied by $|\Omega|$. A rigorous proof of the upper bound  
\begin{equation}\label{eq:upperbound} \limsup_{ n\to \infty}  \frac{n}{\log n} \EE\sqa{ W_2^2\bra{ \frac 1 n \sum_{i=1}^n \delta_{X_i}, \frac 1 n \sum_{i=1}^n \delta_{Y_i}}} \le \frac{|\Omega|}{2 \pi},\end{equation}
has been obtained in  \cite[Theorem 1]{BeCa} under the hypothesis that  $\Omega = (0,1)^2$ and that the common law of the points is absolutely continuous with a uniformly strictly positive and bounded Lipschitz density. Further conjectures on higher order terms in the asymptotic expansion of \eqref{eq:caracciolo} can be found in \cite{benedetto2021random}. Unfortunately, a rigorous mathematical justification  of these predictions  seems currently out of reach.

\subsection*{Main results}
Our main results fully settle the validity of \cite[Conjecture 2]{BeCa}, allowing in fact for densities that are not necessarily smooth, but only H\"older continuous and strictly positive on a bounded connected domain with Lipschitz boundary. We state and prove separately first the case of the quadratic Wasserstein distance between an empirical measure and the common law of the sampled points,  and then that of optimal transport between two independent samples. In the former setting we have the following

\begin{theorem}\label{thm:main-ref}
Let $\Omega \subseteq \R^2$ be a bounded connected domain, with Lipschitz boundary and  $\rho$ be a H\"older continuous probability density on $\Omega$ uniformly strictly positive and bounded from above. Given i.i.d.\ random variables $(X_i)_{i=1}^\infty$ with common distribution $\rho$, we have 
\[ \lim_{n \to \infty} \frac{n}{\log n} \mathbb{E} \sqa{ W_2^2\bra{ \frac 1 n \sum_{i=1}^n \delta_{X_i}, \rho }}  = \frac {|\Omega|} {4 \pi}.\]
\end{theorem}

In the latter case, we also allow for possibly a different number of points $n$, $m$ among two (jointly) i.i.d.\ families of random variables, extending \cite[Remark 7.1]{AmGlau}. 

\begin{theorem}\label{thm:main-ref-bip}
Let $\Omega \subseteq \R^2$ be a bounded connected domain with Lipschitz boundary and  $\rho$ be a H\"older continuous probability density on $\Omega$ uniformly strictly positive and bounded from above. Given i.i.d.\ random variables $(X_i, Y_i)_{i=1}^\infty$ with common distribution $\rho$, for every $q \in [1,\infty]$, 
\[ \lim_{\substack{n,m \to \infty \\ m/n \to q}} \frac{n}{\log n} \mathbb{E} \sqa{ W_2^2\bra{ \frac 1 n \sum_{i=1}^n \delta_{X_i}, \frac 1 m \sum_{j=1}^m \delta_{Y_j} }}  = \frac {|\Omega|} {4 \pi}\bra{ 1 + \frac{1}{q}}.\]
\end{theorem}
Building on these results for Euclidean planar domains, it is possible using similar sub-additivity  arguments to generalize them  to some metric measure spaces, including e.g., smooth connected domains
of two dimensional Riemannian manifolds with densities, see Section~\ref{sec:manifolds} for more precise statements.

Notice that when $q \to \infty$ we obtain in Theorem \ref{thm:main-ref-bip} the same limit as in Theorem~\ref{thm:main-ref}. When $n=m$, Theorem \ref{thm:main-ref-bip} gives the asymptotic value of the 
 minimum bipartite matching. In fact, using simple upper and lower bounds we can  obtain a related result for $n$, $m \to \infty$ with $m-n$ that does not grow too fast. 

\begin{corollary}\label{cor:assignment}
Let $\Omega \subseteq \R^2$ be a bounded connected domain, with Lipschitz boundary and  $\rho$ be a H\"older continuous probability 
density on $\Omega$ uniformly strictly positive and bounded from above.  Given i.i.d.\ random variables $(X_i, Y_i)_{i=1}^\infty$ with common distribution $\rho$, for any sequence $m= m(n) \ge n$ with $\lim_{n \to \infty} (m - n)/ \log n = 0$, it holds
\begin{equation*}\label{eq:ass-limit}  \lim_{n \to \infty } \frac{1}{\log n } \mathbb{E} \sqa{\min_{\sigma \in \mathcal{S}_{n,m}} \sum_{i=1}^n |X_i - Y_{\sigma(i)}|^2
}  = \frac {|\Omega|} {2 \pi}.\end{equation*}
\end{corollary}

We further deduce a result for the average cost of the random quadratic Euclidean bipartite travelling salesperson problem. The Euclidean travelling salesperson problem searches for the shortest cycle that visits a given set of points in $\R^d$, and the study of its random version is  classical \cite{beardwood1959shortest, redmond1996asymptotics, yukich2000asymptotics, yukich1995asymptotics}. Its random bipartite variant requires that the cycle must alternatively connect points from two given sets of points. By extending the argument from \cite{capelli2018exact} from the case of the square to a general domain, we deduce that the cost of the random quadratic Euclidean bipartite travelling salesperson problem in two dimensions is asymptotically twice that of the bipartite matching problem.

\begin{corollary}\label{cor:main-TSP}
Let $\Omega \subseteq \R^2$ be a bounded connected domain, with Lipschitz boundary and $\rho$ be a H\"older continuous probability density on $\Omega$ uniformly strictly positive and bounded from above. Given i.i.d.\ random variables $(X_i, Y_i)_{i=1}^\infty$ with common distribution $\rho$, then
\[ \lim_{n \to \infty} \frac{1}{\log n} \mathbb{E} \sqa{  \min_{\sigma,\tau \in \mathcal{S}_n} \sum_{i=1}^{n} |X_{\sigma(i)} - Y_{\tau(i)}|^2 + |Y_{\tau(i)} - X_{\sigma(i+1)}|^2 } = \frac {|\Omega|} {\pi},\]
where we use the convention  $\sigma(n+1) = \sigma(1)$, for $\sigma \in \mathcal{S}_n$.
\end{corollary}

Finally, let us comment that our main focus on limit results for the expected costs is justified by general concentration arguments that can improve to convergence in probability of the sequence of  renormalized cost -- see \cite[Remark 4.7]{AST}, where the only assumption is the validity of a $L^2$-Poincar\'e-Wirtinger inequality, satisfied on such regular domains.



\subsection*{Idea of the proof} We follow the sub-additivity method introduced in this context in \cite{BaBo,DeScSc13} and later improved in \cite{BeCa, GolTre}. This splits the energy in a local and a
global part. A central point is to prove that the global part, which encodes the defect in sub-additivity, has asymptotically only a vanishing contribution.   While this method has already been  implemented  in \cite{BeCa} to obtain  the upper bound \eqref{eq:upperbound} one of our main achievements is to prove that it may indeed be used also to get the corresponding lower bound. To this aim, we rely on a ``boundary'' variant of the optimal transport which has super-additivity properties. Similar functionals have been widely used in the theory of Euclidean additive functionals \cite{BaBo, Yu98}. The main point is to prove that when $\Omega=(0,1)^2$ and the law is the uniform one,  the asymptotic costs for the usual optimal transport and the ``boundary'' versions are equal (see Proposition \ref{prop:ast-dirichlet} and Proposition \ref{prop:astbip-dirichlet}). Our proof relies on the PDE approach from \cite{AmGlau,AST}. To the best of our knowledge, this is the first example where the so-called Dirichlet-Neumann bracketing method is used in the context of optimal transport (see \cite{ArmSer} for a recent application of these ideas for Coulomb gases).

On a more technical side, we introduce two main ideas. First, in order to deal with domains which are not $(0,1)^2$ or a finite union of cube, we consider a Whitney partition of our domain. While this  does not affect too much the treatment of the local term, it requires a finer estimate of the global one. Indeed,  as in \cite{BeCa,GolTre} (see also \cite{peyre2018comparison,AST,GHO,Le17}) we first  rely on the estimate
\[
 W_2^2(f,1)\les \|f-1\|_{H^{-1}}^2
\]
to reduce ourselves to an estimate in $H^{-1}$. In \cite{BeCa,GolTre} (and also \cite{GHO}) the right-hand side is then estimated thanks to Poincar\'e inequality by an $L^2$ norm. A quick look at \cite{BeCa} shows that for fixed number of points $n$, this yields an error which is proportional to the number of cubes in the partition of our domain. Since a Whitney partition  is made of  infinitely many such cubes (or in any case a very large number of them, see Lemma \ref{lem:decomp}) this would lead to a very bad estimate. This is due to the fact that a function with rapid oscillations typically has large $L^2$ norm but small $H^{-1}$ norm. To capture this, we prove in Lemma \ref{lem:Sob}  a finer estimate for functions with ``small'' support. This gives an error term which, up to logarithm, does not depend on the mesh-size of our partition (see \eqref{eq:global}). 
In order to deal with H\"older continuous densities instead of Lipschitz ones as in \cite{BeCa}, the second idea is to replace the Knothe map used in \cite[Lemma 1]{BeCa} by a transport map constructed via heat flow interpolation in the spirit of the Dacorogna-Moser construction (see Lemma~\ref{lem:heat-cube}). This might be of independence interest.

\subsection*{Further questions} We mention here some open problems related to our results:
\begin{enumerate}
\item In \cite{ambrosio2019optimal} and \cite{goldman2021fluctuation}  quantitative rates of convergence for the optimal transport map have been obtained. It may be worth exploring whether similar quantitative rates still hold in the general setting addressed in this article.
\item  We expect that our results may be extended to the case of  less regular domains and densities, with a similar asymptotic cost. However, the validity of \eqref{eq:upperbound} must require some condition on the support of $\rho$, such as connectedness \cite[Remark 2]{BeCa}, which is in contrast with the universal upper bounds obtained by Talagrand \cite{Ta92} for the case $p=1$.
\item The case of unbounded domains may be treated with similar techniques, possibly leading to sharper limit results for a wide class densities, e.g.\ Gaussian ones \cite{Le17, Le18, talagrand2018scaling}.
\end{enumerate}

\subsection*{Structure of the paper} In Section \ref{sec:intro} we fix notation and provide some general results about negative Sobolev norms, classical optimal transport and its ``boundary'' version and finally a construction of a transport map from a H\"older density to  the uniform one via heat flow on the cube. Section~\ref{sec:wd} is devoted to the extension of the PDE approach  from \cite{AmGlau,AST} to the ``boundary'' transport cost on the square. In Section~\ref{sec:main-ref} and Section~\ref{sec:main-bip} we prove respectively Theorem~\ref{thm:main-ref} and Theorem~\ref{thm:main-ref-bip}. Section~\ref{sec:manifolds} describes how similar techniques allow to consider more general settings, including Riemannian manifolds. Section~\ref{sec:ass} and Section ~\ref{sec:TSP} contain respectively the proof of Corollary~\ref{cor:assignment} and Corollary~\ref{cor:main-TSP}. Two appendices include technical results about decomposition of a domain into sets with certain properties that may be well-known in the literature, but we did not find exactly in a version useful for our purposes.

\section{Notation and preliminary results}\label{sec:intro}

We write $|A|$ for the Lebesgue measure of a Borel set $A \subseteq \R^d$, and $\int_A f$ for the integral of an integrable function $f$ on $A$. 
If a measure is absolutely continuous with respect to the Lebesgue measure, we always identify it with its density. A cube $Q \subseteq \R^d$ of side length $\ell$ is a set of the 
form $Q = \prod_{i=1}^d (v_i, v_i+\ell)$ with $(v_i)_{i=1}^d \in \R^d$ and $\ell>0$. We note $\ell(Q)$ the sidelength of a cube $Q$. By a partition of $\Omega\subseteq \R^d$,
we always mean in fact that $\Omega$ is covered up to Lebesgue negligible sets. We denote by $|v|$ the  Euclidean norm of a vector $v\in \R^d$.  We use the letters $C$, $c$ for   
positive constants whose value may vary from one line to the next and $\omega$ for a generic decreasing rate function with $\lim_{t\to \infty} \omega(t)=0$.
In this paper the domain $\Omega$ is fixed  and we write $A\les B$ if there is  $C>0$ depending only on $\Omega$ (and potentially on $d$ and $p$ if we consider the $p-$Wasserstein distance in $\R^d$) such that $A\le C B$. For a function $f$, we use the notation $\nabla f$ for its gradient, $\nabla\cdot f$ for its divergence, $\Delta f = \nabla \cdot \nabla f$ for its Laplacian. The push-forward of a measure $\mu$ with respect to a map $f$ is denoted $f_\sharp \mu$, i.e., $f_\sharp \mu( A) = \mu(f^{-1}(A))$. 

For $f: \Omega \subseteq \R^d \to \R^k$, we write $\nor{f}_{L^p(\Omega)} = \bra{ \int_{\Omega} |f|^p}^{1/p}$, $\nor{f}_{L^\infty(\Omega)} = \sup_{x \in  \Omega} |f(x)|$ and
\[ \nor{f}_{C^\alpha(\Omega)} = \nor{f}_{L^\infty(\Omega)} + \sup_{ \substack{x,y\in \Omega \\ x\neq y }} \frac{ |f(y)-f(x)|}{|x-y|^\alpha},\]
for $\alpha \in (0,1]$ for the $\alpha$-H\"older norm of $f$ (when $\alpha=1$ we obtain its Lipschitz norm). We also write
\[ \Lip_{\Omega} f = \sup_{ \substack{x,y\in \Omega \\ x\neq y }} \frac{ |f(y)-f(x)|}{|x-y|}.  \]
We omit to specify $\Omega$ when it is clear from the context and write $\nor{f}_{p}$, $\nor{f}_\infty$, $\nor{f}_{C^\alpha}$ and $\Lip f$.

We collect below some general facts useful to prove our main results. For possible future reference we state and prove them in a slightly more general form, e.g., in general dimensions or for general, non quadratic, costs.

\subsection{Negative Sobolev norms}

Let $\Omega \subseteq \R^d$ be a bounded connected open set with Lipschitz boundary. We denote the negative Sobolev norm by 
\[
 \|f\|_{W^{-1,p}(\Omega)}=\sup_{|\nabla \phi|_{L^{q}(\Omega)}\le 1} \int_{\Omega} f \phi,
\]
where $q$ is the H\"older conjugate of $p$ (i.e. $1=\frac{1}{p}+\frac{1}{q}$). For $p=2$ we simply write $H^{-1}(\Omega)$. Notice in particular that in order to have $\|f\|_{W^{-1,p}(\Omega)}<\infty$ we must have $\int_\Omega f=0$. In this case we may also restrict the supremum to functions $\phi$ having also average zero. When it is clear from the context, we will drop the explicit dependence on  $\Omega$ in the norms.

The heat semi-group with null Neumann boundary conditions on $\Omega$ is well-defined as the symmetric Markov semi-group of operators $(P_t)_{t \ge 0}$ 
arising as the $L^2$ gradient flow of the Dirichlet energy $\nor{\nabla f}^2_2$ on the Sobolev space $f \in H^1(\Omega)$. It is well-known \cite{varopoulos2008analysis} that the validity of $L^2$-Poincar\'e-Wirtinger inequality on $\Omega$ is equivalent
to a spectral gap, i.e., for some constant $c>0$, for every $f \in L^2(\Omega)$ with $\int_{\Omega} f = 0$,
\begin{equation}\label{eq:exp-contractivity} \nor{ P_t f }_{2} \le e^{-c t} \nor{f}_{2}, \quad \text{for $t \ge 0$.}\end{equation}
In particular, one has the integral representation
\[ \Delta^{-1} f = \int_0^\infty  P_t f d t\]
for the solution to the elliptic PDE, $\Delta u= f$ in $\Omega$ with null Neumann boundary conditions and also, for the negative Sobolev norm, 
\begin{equation}\label{eq:negative-sobolev} \nor{f}_{H^{-1}}^2 =  \int_{\Omega}  f \Delta^{-1} f  = \int_0^\infty  \int_\Omega f P_t f dt = \int_0^\infty  \int_\Omega (P_{t/2} f)^2 dt.\end{equation}
Sobolev inequality instead is equivalent to ultracontractivity: for every $t \in [0,1]$ and $f \in L^1$ with $\int_{\Omega}f = 0$,
\begin{equation}\label{eq:UC} \nor{ P_t f }_{\infty} \les t^{-d/2} \nor{f}_{1}.\end{equation}  In 
terms of the symmetric heat kernel $p_t(x,y) = p_t(y,x)$ (defined by  the equality $P_tf(x)=\int_\Omega p_t(x,y)f(y) dy$), it reads $\nor{p_t}_\infty \les t^{-d/2}$. This inequality can be easily combined with \eqref{eq:exp-contractivity} to obtain the stronger inequality 
\begin{equation}\label{eq:stronger-ultra} \nor{ P_t f}_{\infty} \les t^{-d/2} e^{-c t} \nor{f}_{1}.\end{equation}

The following lemma will be used to bound the negative Sobolev norm of functions supported on subsets $A \subseteq \Omega$. 

\begin{lemma}\label{lem:Sob}
Let $d\ge 2$, $\Omega \subseteq \R^d$ be a bounded connected Lipschitz domain. For every $f \in L^\infty(\Omega)$ with $\int_\Omega f=0$,
\[ \lt\| f \rt\|_{H^{-1}} \lesssim \begin{cases} \nor{f}_{L^1} \sqrt{ \abs{ \log \frac{\nor{f}_{1}}{\nor{f}_\infty}} + 1}& \text{if $d=2$,}\\
 \nor{f}_{1}^{\frac{1}{2}+\frac{1}{d}}\nor{f}_{\infty}^{\frac{1}{2}-\frac{1}{d}} & \text{if $d>2$.}\end{cases}\]
 \end{lemma}
 \begin{proof}
 By \eqref{eq:negative-sobolev} 
 \[ \nor{ f}_{H^{-1}}^2 = \int _0^\infty \int_{\Omega}  f P_t   f dt.\] 
By \eqref{eq:stronger-ultra} and $\|P_t f\|_\infty\le \|f\|_\infty$, for every $t>0$,
 \[ \int_\Omega f P_t  f  \le \nor{ f}_{1} \nor{ P_t f}_{\infty} \les \nor {f}_1 \min(\nor{f}_\infty, t^{-\frac{d}{2}} \exp(-ct) \nor {f}_1).\]
 By integration  we obtain for every $t_0>0$,
\[  \nor{ f}_{H^{-1}}^2 \lesssim  \nor{f}_{1}\lt( t_0\nor{f}_\infty +  \nor{f}_{1} \begin{cases} |\log t_0| & \text{ if $d=2$,}\\
t_0^{1-d/2} & \text{if $d>2$}\end{cases}\rt).\]
Optimizing in $t_0$ by setting $t_0=\nor{f}_1^{d/2} \nor{f}_\infty^{-d/2}$ concludes the proof.
 \end{proof}

\subsection{Optimal transport}

We introduce some notation for the Wasserstein distance and recall few simple properties that will be used throughout.
Proofs can be found in any of the monographs \cite{Viltop, Santam}. 

Let $p \ge 1$ and $\mu$, $\lambda$  be positive Borel measures with finite $p$-th moments and equal mass $\mu(\R^d) = \lambda(\R^d)$. The Wasserstein distance of order $p\ge 1$ between $\mu$ and $\lambda$ is defined as
\[ W_{p}(\mu, \lambda) = \left(\min_{\pi\in\mathcal{C}(\mu,\lambda)} \int_{\R^d\times\R^d} |x-y|^p d \pi(x,y)\right)^\frac{1}{p},\]
where $\mathcal{C}(\mu, \lambda)$ is the set of couplings between $\mu$ and $\lambda$. For a Borel subset $\Omega \subseteq \R^d$, we also write
\[  W_{\Omega, p}(\mu, \lambda) = W_p (\mu \restr \Omega, \lambda \restr \Omega),\]
which implicitly assumes that  $\mu(\Omega) = \nu(\Omega)$.

Let us recall that $W_p$ is a distance, in particular the triangle inequality holds:
\[
W_{p}( \mu, \nu) \le W_{p}(\mu, \lambda) + W_{p}(\nu, \lambda).
\]
We will also  use the classical sub-additivity inequality 
\begin{equation}\label{eq:sub} W_{p}^p\bra{ \sum_{k} \mu_k, \sum_{k} \lambda_k} \le \sum_k W^p_{p}(\mu_k, \lambda_k),\end{equation}
for a finite set of positive measures $\mu_k$, $\lambda_k$. A simple combination of the two properties above and Young inequality yield a geometric subadditivity property 
\cite[Lemma 3.1]{GolTre}, that we report here for the reader's convenience: there exists a constant $C = C(p)>0$ such that, for every Borel partition
$(\Omega_k)_{k\in \N}$ of $\Omega$, and every $\eps\in(0,1)$, 
 \begin{equation}\label{eq:mainsub}
  W^p_{\Omega,p}\lt(\mu, \frac{\mu(\Omega)}{\lambda(\Omega)}\lambda\rt)\le (1+\eps)\sum_k W^p_{\Omega_k,p}\lt(\mu,\frac{\mu(\Omega_k)}{\lambda(\Omega_k)}\lambda\rt)+ 
  \frac{C}{\eps^{p-1}}W^p_{\Omega,p}\lt(\sum_k \frac{\mu(\Omega_k)}{\lambda(\Omega_k)}\chi_{\Omega_k}\lambda,\frac{\mu(\Omega)}{\lambda(\Omega)}\lambda\rt).
 \end{equation}

 If we assume that $\Omega$ is sufficiently regular, then we may use the Benamou-Brenier formula in a similar fashion as in \cite[Lemma 3.4]{GolTre} (see also \cite[Corollary 3]{peyre2018comparison}) to bound from above the Wasserstein distance by the negative Sobolev norm. Notice that we use the fact that the Euclidean distance is bounded from above by the geodesic distance in $\Omega$ (they coincide if $\Omega$ is convex). 
 
 \begin{lemma}\label{lem:peyre} Assume that $\Omega \subseteq \R^d$ is a bounded connected open set with  Lipschitz boundary. If $\mu$ and $\lambda$ are measures on $\Omega$ with $\mu(\Omega) = \lambda(\Omega)$,  absolutely continuous with respect to the  Lebesgue measure and $\inf_\Omega \lambda>0$, then, for every $p\ge 1$,
 \begin{equation}\label{eq:estimCZ}
  W_{\Omega, p}^p(\mu,\lambda)\les \frac{1}{\inf_{\Omega} \lambda^{p-1}}\nor{ \mu - \lambda}_{W^{-1,p}(\Omega)}^p.
 \end{equation}
 \end{lemma}

To deal with lower bounds, we rely on a ``boundary'' variant of the optimal transport, introduced in \cite{FigGig}, but also independently studied in the setting of Euclidean bipartite matching problems, see \cite{BaBo, DeScSc13}. Here, one is allowed to transport any amount of mass from and to the boundary of an open set $\Omega \subseteq \R^d$. We write, for finite positive measures $\mu$ and $\nu$ with finite $p$-th moment,
\[
 \Wb_{\Omega, p}^p(\mu,\lambda)=\bra{ \min_{ \pi\in \mathcal{C}b_{\Omega}(\mu, \lambda)} \int_{\overline{\Omega}\times\overline{\Omega}} |x-y|^p d \pi(x,y)}^{\frac 1 p},
\]
where $\mathcal{C}b_{ \Omega}(\mu, \lambda)$ is the set of positive measures $\pi$ on $\overline{\Omega}\times\overline{\Omega}$ such that $\pi_1\restr \Omega=\mu\restr \Omega$ and $ \pi_2\restr\Omega=\lambda\restr \Omega$ ($\pi_1$, $\pi_2$ denoting respectively the first and second marginals of $\pi$).

The triangle inequality holds,
\[ \Wb_{\Omega, p}( \mu, \nu) \le \Wb_{\Omega, p}(\mu, \lambda) + \Wb_{\Omega, p}(\nu, \lambda).\]
A geometric superadditivity property also holds: for every disjoint family $(\Omega_k)_{k\in \N}$ of open subsets of $\Omega$,
\[  \Wb_{\Omega, p}^p( \mu, \nu)  \ge \sum_{k} \Wb_{\Omega_k, p}^p( \mu, \nu). \]
The intuition behind this property is quite clear. Given any plan on $\Omega$, one should be able to suitably restrict it on each $\Omega_k$ by stopping the transport at each boundary. 
To prove it rigorously, it is sufficient to argue in the case of discrete measures $\mu = \sum_{i} \mu_i \delta_{x_i}$, $\nu = \sum_{i} \nu_i \delta_{y_i}$. The general case follows by approximation. Given $\pi \in \mathcal{C}b_{\Omega,p}(\mu, \nu)$, we define $(\pi_k)_{k \in \mathbb{N}}$ with $\pi_k \in \mathcal{C}b_{ \Omega_k,p}(\mu, \nu)$ such that
\begin{equation}\label{eq:super-proof} \int_{\overline{\Omega}\times\overline{\Omega}} |x-y|^p d \pi(x,y) \ge \sum_{k} \int_{\overline{\Omega_k}\times\overline{\Omega_k}} |x-y|^p d \pi_k(x,y).
\end{equation}
Again by approximation, we may also assume that $\pi = \sum_{\ell} \pi_\ell \delta_{(x_\ell, y_\ell)}$ is discrete, with $(x_{\ell}, y_{\ell}) \in \overline{\Omega}\times \overline{\Omega}$. 
The following algorithm can be used to define the sequence $(\pi_k)_{k}$. Set initially $\pi_k= 0$ for  every $k$. For every $\ell$, consider the pair $(x_\ell, y_\ell)$.
If $x_\ell,y_\ell\in \Omega_k$, then add $\pi_\ell \delta_{(x_\ell,y_\ell)}$ to $\pi_k$. If $x_\ell \in \Omega_k$ and $y_\ell \in \Omega_j$ with $k \neq j$, let
\[ t^- = \inf \cur{ t \in [0,1]: (1-t)x_\ell +ty_\ell \notin \Omega_k}\]
\[ t^+ = \inf \cur{t \in [0,1]: tx_\ell+ (1-t)y_\ell \notin \Omega_j}, \]
and define $y^-_\ell = (1-t^-)x_\ell+t^-y_\ell\in \partial{\Omega_k}$, $x^+_\ell = t^+x_\ell+ (1-t^+)y_\ell \in \partial{\Omega_j}$. Then, add respectively $\pi_\ell \delta_{(x_\ell, y^-_\ell)}$ to $\pi_k$ and $\pi_\ell \delta_{(x^+_\ell,y_\ell)}$ to $\pi_j$. 
If $x \in  \Omega_k$ for some $k$ and $y \in \overline\Omega \setminus \bigcup_{j}\Omega_j$, set  $t^-$ and $y^-_\ell$ as the previous case and add $\pi_\ell \delta_{(x_\ell, y^-_\ell)}$ to $\pi_k$.
Similarly, if $x   \in \overline\Omega \setminus \bigcup_{j}\Omega_j$ and $y \in \Omega_k$ then  add $\pi_\ell \delta_{(x^+_\ell, y)}$ to $\pi_k$.
In all the other cases, i.e., if $x_\ell,y_\ell \in \overline\Omega \setminus \bigcup_{k} \Omega_k$, do nothing. It is not difficult to check that $(\pi_k)_{k}$
thus defined indeed satisfies $\pi_k \in \mathcal{C}b_{ \Omega_k}$ and  \eqref{eq:super-proof} holds (using in particular that $p \ge 1$ to argue that
$|x_\ell-y_\ell|^p \ge |x_\ell-y^-_\ell|^p + |x^+_\ell-y_\ell|^p$).

Arguing similarly as in the proof of \eqref{eq:mainsub} we obtain a symmetric inequality: there exists a constant $C(p)>0$ such that, for every Borel partition $(\Omega_k)_{k\in \N}$ 
of $\Omega$, and every $\eps\in(0,1)$, 
 \begin{equation}\label{eq:mainsup}
 \begin{split} 
   \Wb^p_{ \Omega,p}\lt(\mu, \frac{\mu(\Omega)}{\lambda(\Omega)}\lambda\rt)   \ge & (1-\eps)\sum_k \Wb^p_{\Omega_k,p}\lt(\mu,\frac{\mu(\Omega_k)}{\lambda(\Omega_k)}\lambda\rt)\\
   & -\frac{C}{\eps^{p-1}}\Wb^p_{\Omega,p}\lt(\sum_k \frac{\mu(\Omega_k)}{\lambda(\Omega_k)}\chi_{\Omega_k}\lambda,\frac{\mu(\Omega)}{\lambda(\Omega)}\lambda\rt).
   \end{split}
 \end{equation}
 
To obtain lower bounds, we will also rely on the dual formulation of $\Wb_{\Omega, p}$ that reads
\begin{equation}\label{eq:duality} 
\begin{split} & \Wb_{\Omega, p}^p(\mu, \nu) =\\
  & \sup \cur{ \int_\Omega f d \mu + \int_\Omega g d \nu\, : \, \text{ $f(x)+g(y) \le |x-y|^p$ for $x$, $y \in \Omega$, $f = g = 0$ on $\partial \Omega$}}.
 \end{split}
\end{equation}
In fact, only the inequality $\ge$ will be used, which is immediate to check.

\begin{remark}
Many of the above properties, in particular \eqref{eq:mainsub} and \eqref{eq:mainsup}  hold for the Wasserstein distance defined over any length metric space, in particular on a Riemannian manifold.
\end{remark}

\subsection{A transport map via heat flow} 
The following result yields a Lipschitz map on the cube $(0,1)^d$  transporting a given H\"older probability density to the uniform one. The main point here is that the Lipschitz constant of the map is very close to $1$.  This construction, possibly interesting on its own, allows us to avoid the use of general boundary regularity theory for the optimal transport map with respect to the squared distance cost \cite{caffarelli1996boundary}, as it does not seem to be applicable to the case of the cube (it requires strictly convex and $C^2$ boundary). Notice that for H\"older continuous densities, counterexamples can be  constructed if we work with the Knothe map as in \cite{BeCa}. 

\begin{proposition}\label{prop:map-heat-semigroup}
For $d\ge 1$, $\alpha \in (0,1]$, there exists $C>0$ depending on $d$ and $\alpha$ only such that the following holds: for any $\rho: (0,1)^d \to (0, \infty)$ with
\[ \int_{(0,1)^d} \rho = 1 \quad \text{and} \quad  \nor{\rho -1}_{C^\alpha} \le 1/2,\]
there exists $T: (0,1)^d \to (0,1)^d$ such that $T_{\sharp} \rho = 1$, with
\[ \Lip T, \Lip T^{-1} \le  1 + C \nor{ \rho -1}_{C^\alpha} \quad \text{and} \quad  T(\partial(0,1)^d) = \partial(0,1)^d.\]
\end{proposition}

\begin{proof}
Using the Neumann heat semi-group, we define $\rho_t = P_t \rho$, so that, for $t \ge 0$, the weak formulation of the heat equation reads, for every for every $f \in H^1((0,1)^d)$, 
\begin{equation*}\label{eq:neumann-weak} \int_{(0,1)^d} f \rho_t  = \int_{(0,1)^d} f \rho - \int_0^t \int_{(0,1)^d} \nabla f  \, \nabla \rho_t d t, \quad \text{for every $t \ge 0$.} \end{equation*}
We notice first that the assumption gives $\inf_{(0,1)^d} \rho \ge 1/2$ so that, for every $t \ge 0$, $\inf_{(0,1)^d} \rho_t \ge 1/2$ as well, hence
\[ \nor{ \rho_t^{-1} }_{\infty} \les 1.\]
By standard heat kernel estimates (one has in fact an explicit representation of $p_t(x,y)$, see \cite[Appendix A]{AmGlau} for the case $d=2$) we have, for every $t>0$, 
\[ \nor{ \nabla \rho_t }_{\infty} \les t^{-\frac 1 2  + \frac{\alpha}{2}} e^{-ct} \nor{ \rho-  1}_{C^\alpha}\quad \text{and} \quad \nor{ \nabla^2 \rho_t }_{\infty} \les t^{-1  + \frac{\alpha}{2}} e^{-ct} \nor{ \rho-  1}_{C^\alpha}.\]
We thus define the time-dependent vector field $(b_t)_{t >0}$, on $(0,1)^d$,
\[ b_t(x) = - \nabla \log \rho_t(x) = -\frac{\nabla \rho_t(x)}{\rho_t(x)}, \quad \text{with} \quad  \nabla b_t(x) = - \frac{\nabla^2 \rho_t(x)}{\rho_t(x)} + \frac{\nabla \rho_t(x) \otimes \nabla \rho_t(x)}{\rho_t^2(x)}.\]
Using the previous estimates, it follows that
\[ \nor{ b_t}_{\infty} \les t^{-\frac 1 2  + \frac{\alpha}{2}} e^{-ct} \nor{ \rho-  1}_{C^\alpha} \quad \text{and} \quad \nor{ \nabla b_t}_{\infty} \les t^{-1  + \frac{\alpha}{2}} e^{-ct}  \nor{ \rho-  1}_{C^\alpha} .\]
Since the right-hand sides  are integrable functions with respect to $t \in (0, \infty)$,  standard arguments  yield that the associated flow $X(t,x)$, i.e., the solution to
\[ X(t,x) = x+ \int_0^t b_s(X(s,x)) d s \quad \text { for $t \ge 0$, $x \in (0,1)^d$}\]
is well-defined and Lipschitz continuous:
\begin{equation}\label{eq:lip-t} \Lip X(t,\cdot) \le \exp\bra{ \int_0^t \nor{\nabla b_s}_\infty d s } \le \exp\bra{  C  t^{\frac{\alpha}{2}} e^{-ct}  \nor{ \rho-  1}_{C^\alpha} }.\end{equation}
Because of  the homogeneous Neumann boundary condition, $b_t$ is tangent to the boundary $\partial (0,1)^d$, hence $X(t,x) \in (0,1)^d$ for every $t \ge 0$. Moreover, $\lim_{t \to \infty} X(t,x) \in (0,1)^d$ exists for every $x \in (0,1)^d$, since, for $s \le t$,
\[ \abs{ X(t,x) - X(s,x) } \le  \int_s^t\nor{ b_r}_\infty d r \les s^{\frac{\alpha}{2}}e^{-cs} \nor{ \rho-  1}_{C^\alpha},\]
which yields that $t \mapsto X(t,x)$ is Cauchy as $t \to \infty$.  We thus define $T(x) = \lim_{t \to \infty}X(t,x)$. Letting $t \to \infty$ in \eqref{eq:lip-t}, we have
\[  \Lip T  \le \exp\bra{ C   \nor{ \rho-  1}_{C^\alpha} } \le 1 + C\nor{ \rho-  1}_{C^\alpha}.\]
 To prove that $T$ is invertible, we simply notice that for every $t >0$ the inverse map of $x \mapsto X(t,x)$ is given by $y \mapsto Y^t(t,y)$, $Y^t$ being the flow of the ``backward'' in time vector field $(b^{t}_s)_{s \in [0,t]}$, $b^t_s(y) = -b_{(t-s)}(y)$. Writing the analogue of \eqref{eq:lip-t} for $y \mapsto Y^t(t,y)$ and letting $t \to \infty$ yields the claimed bound on $\Lip T^{-1}$. 

To conclude, we need to argue that $T_\sharp \rho = 1$.
This follows quite classically from the fact that $\tilde \rho_t = X(t, \cdot)_{\sharp} \rho$ and $\rho_t$ both solve the same continuity equation for the vector field $(b_t)_{t>0}$.
Since $b_t$ is locally Lipschitz continuous with integrable in time Lipschitz norm, 
uniqueness holds for the Cauchy problem (this follows for example by a slight adaptation of  \cite[Theorem 4.4]{Santam}). 
Therefore, $\rho_t = \tilde \rho_t$ and $T_\sharp \rho = \lim_{t \to \infty} \tilde \rho_t = \lim_{ t \to \infty} \rho_t = 1$. \end{proof}

In general, if $\Omega \subseteq \R^d$ is open and $T: \Omega \to \Omega$ is Lipschitz, then for any pair of measures $\mu$, $\lambda$  with $\mu(\Omega) = \lambda(\Omega)$,
\begin{equation}\label{eq:W-T} W_{\Omega,p} (T_\sharp (\mu \restr \Omega), T_\sharp(\lambda \restr \Omega)) \le (\Lip T)  W_{\Omega,p} (\mu, \lambda),\end{equation}
since any coupling $\pi$ between $\mu \restr \Omega$ and $\lambda \restr \Omega$ induces the coupling $(T,T)_{\sharp} \pi$. For the ``boundary'' optimal transport a similar inequality holds provided that $T: \overline{\Omega} \to \overline{\Omega}$ is such that $T(\Omega) \subseteq \Omega$ and $T(\partial \Omega) \subseteq \partial \Omega$:
\begin{equation}\label{eq:Wb-T} \Wb_{\Omega,p} (T_\sharp (\mu \restr \Omega), T_\sharp(\lambda \restr \Omega)) \le (\Lip T)  \Wb_{\Omega,p} (\mu, \lambda).\end{equation}
Combining these observations with the Proposition \ref{prop:map-heat-semigroup}, we obtain the following result for the bipartite matching problem, extending \cite[Lemma 1]{BeCa} to the case of H\"older continuous densities. 

\begin{lemma}\label{lem:heat-cube} Let $d \ge 1$, $\alpha \in (0,1]$, $p \ge 1$ and let $\rho$ be  a uniformly positive and $\alpha$-H\"older
continuous function on a domain $\Omega \subseteq \R^d$. Then, there exists a constant $c = C( p, \rho)>0$ such that, for every cube $Q \subseteq \Omega$ with side length $r<c$, the following holds:
\begin{enumerate}
\item if $(X_i)_{i=1}^n$ are  i.i.d.\ on $Q$ with common density $\rho/\rho(Q)$ and $(X_i^r)_{i=1}^n$ are i.i.d.\ uniformly distributed on $[0,r]^d$, then
\begin{multline}\label{eq:holder-cube} 
 \lt|\EE\sqa{ W_*^p\bra{   \frac 1 n \sum_{i=1}^n \delta_{X_i}, \frac{\rho}{\rho(Q)} }}- \EE\sqa{ W_{*}^p\bra{  \frac 1 n \sum_{i=1}^n \delta_{X_i^r}, \frac{1}{|Q|} }}\rt|\\
 \les r^\alpha \EE\sqa{ W_{*}^p\bra{    \frac 1 n \sum_{i=1}^n \delta_{X_i^r}, \frac{1}{|Q|} }},
\end{multline}
where  $W_*\in\{W_{Q,p},\Wb_{Q,p}\}$;
\item with the same notation, if $(X_i)_{i=1}^n$, $(Y_j)_{j=1}^m$ are (jointly) i.i.d.\ on $Q$ with common density $\rho/\rho(Q)$ and $(X_i^r)_{i=1}^n$, $(Y^r_j)_{j=1}^m$ are (jointly) i.i.d.\ and uniformly distributed on $[0,r]^d$, then
\begin{multline}\label{eq:holder-cube-mn} 
 \lt|\EE\sqa{ W_*^p\bra{   \frac 1 n \sum_{i=1}^n \delta_{X_i}, \frac 1 m \sum_{j=1}^m \delta_{Y_j} } }- \EE\sqa{ W_{*}^p\bra{  \frac 1 n \sum_{i=1}^n \delta_{X_i^r}, \frac 1 m \sum_{j=1}^m \delta_{Y_j^r} }}\rt|\\
 \les r^\alpha \EE\sqa{ W_{*}^p\bra{    \frac 1 n \sum_{i=1}^n \delta_{X_i^r}, \frac 1 m \sum_{j=1}^m \delta_{Y_j^r} }}.
\end{multline}
\end{enumerate}

\end{lemma}

\begin{proof}
The proof is identical to the proof of \cite[Lemma 1]{BeCa} but we include it for the reader's convenience. We only prove that for $r$ small enough,
\[
(1-C r^\alpha)\EE\sqa{ \Wb_{*}^p\bra{  \frac 1 n \sum_{i=1}^n \delta_{X_i^r}, \frac{1}{|Q|} }} \le  \EE\sqa{ \Wb_{*}^p\bra{  \frac 1 n \sum_{i=1}^n \delta_{X_i}, \frac{\rho}{\rho(Q)} }}
\]
since the other inequalities may be proven in the same way. Without loss of generality we assume $Q = (0,r)^d$. Let also $\bar{\rho} = \min_{\Omega} \rho$. 
We define $\rho^r(x) = \rho(rx)r^d/\rho(Q)$ for $x \in (0,1)^d$, so that $\int_{(0,1)^d} \rho^r = 1$, and for every $x$, $y\in (0,1)^d$,
\[ \rho^r(x) - \rho^r(y) \le \frac{ \nor{\rho}_{C^\alpha}}{\bar{\rho}} r^\alpha |x-y|^\alpha,\]
thus $\nor{ \rho^r - 1 }_{C^\alpha}\le 1/2$ if $r$ is sufficiently small.
 We define $S: \overline{Q} \to \overline{Q}$ as  $S(x) = r T(x/r)$, where $T$ is the map provided by Proposition~\ref{prop:map-heat-semigroup}.
 It holds $\Lip S = \Lip T$ and $\Lip S^{-1} = \Lip T^{-1}$,  $S_\sharp\rho/\rho(Q) = 1/|Q|$ is uniform, hence the variables $(S(X_i))_{i=1}^n$ have the same law as $(X_i^r)_{i=1}^n$. Moreover, 
\[ S_\sharp \bra{ \frac{ 1}{n} \sum_{i=1}^n \delta_{X_i}} = \frac 1 n \sum_{i=1}^n \delta _{S(X_i)}\]
has the same law as $\frac 1 n \sum_{i=1}^n \delta_{X_i^r}$. Since $S(\partial Q) = \partial Q$ and $S(Q) = Q$, it follows from \eqref{eq:W-T} and  \eqref{eq:Wb-T} that
\[   \EE\sqa{ \Wb_*^p\bra{  \frac 1 n \sum_{i=1}^n \delta_{X_i^r}, \frac{1}{|Q|} }} \le (\Lip T)^p \EE\sqa{ \Wb_*^p\bra{  \frac 1 n \sum_{i=1}^n \delta_{X_i}, \frac{\rho}{\rho(Q)} }},\]
This proves the claim provided $r$ is   small enough so that  $(\Lip T)^p \le (1-C r^\alpha)^{-1}$. 
\end{proof}

\section{Asymptotic equivalence of the usual and boundary costs for uniform points on the square}\label{sec:wd}

In this section we show how to  modify the proofs from \cite{AST}, using also ideas from \cite{AmGlau}, to obtain the analogue for $\Wb$ of the limits
\begin{equation}\label{eq:ast-cube-ref} \lim_{n \to \infty} \frac{n}{\log n} \EE\sqa{ W_{2}^2 \bra{\frac 1 n \sum_{i=1}^n \delta_{X_i}, 1 }} = \frac 1 {4 \pi},\end{equation}
and
\[ \lim_{n \to \infty} \frac{n}{\log n} \EE\sqa{ W_{2}^2 \bra{\frac 1 n \sum_{i=1}^n \delta_{X_i}, \frac 1 n \sum_{i=1}^n \delta_{Y_i} }} = \frac 1 {2 \pi}, \]
where $(X_i, Y_i)_{i=1}^\infty$ are i.i.d.\ uniformly distributed on $(0,1)^d$. We deal first with the simpler case of matching to the reference measure.

\begin{proposition}\label{prop:ast-dirichlet}
Let $(X_i)_{i =1 }^\infty$ be i.i.d.\ uniformly distributed on $(0,1)^2$. Then
\[ \lim_{n \to \infty} \frac{n}{\log n} \EE\sqa{ \Wb_{(0,1)^2, 2}^2 \bra{\frac 1 n \sum_{i=1}^n \delta_{X_i}, 1 }} = \frac 1 {4 \pi}.\]
\begin{proof}
To simplify the notation, we write $\Wb$ instead of $\Wb_{(0,1)^2, 2}$ and $W$ instead of $W_{(0,1)^2 ,2}$. Since $\Wb \le W$, it is sufficient to prove that
\begin{equation}\label{eq:lower-bound-square} \liminf_{n \to \infty} \frac{n}{\log n} \EE\sqa{ \Wb^2 \bra{ \frac 1 n \sum_{i=1}^n \delta_{X_i}, 1}} \ge  \frac 1 {4 \pi}.\end{equation}
We do it in several steps. 

\emph{Step 1 (Regularization)}.
Write $\mu_n = \frac 1 n \sum_{i=1}^n \delta_{X_i}$, and for $t \ge 0$,
\[ \mu^{t}_n(x) = P_t \mu_n =\frac{1}{n} \sum_{i=1}^n p_t(x,X_i),\]
where we recall that $(P_t)_{t \ge 0}$ denotes the Neumann heat semi-group on $(0,1)^d$ with kernel $p_t(x,y)$. The triangle inequality for $\Wb$ and the inequality $\Wb \le W$ give
\begin{equation*}
\begin{split} \Wb^2 \bra{\mu^t_n, 1} & \le \bra{\Wb(\mu_{n}, 1)+ W(\mu_{n}, \mu^t_n )}^2 \\
& \le  \Wb^2(\mu_{n}, 1) +  W(\mu_n, \mu^t_n) \bra{W(\mu_n, \mu^t_n) + 2 W(\mu_{n}, 1)}.
\end{split}
\end{equation*}
We choose  $t = (\log n)^4/n$, so that, the refined contractivity estimate \cite[Theorem 5.2]{AmGlau} gives
\[ \EE\sqa{ W^2(\mu_n, \mu^t_n)} \lesssim \frac{ \log\log n}{n}.\]
Together with the upper bound 
\[ \EE\sqa{ W^2(\mu_{n}, 1)} \lesssim \frac{\log n}{n},\]
which follows from \eqref{eq:ast-cube-ref}, we obtain that for $n$ large enough,
\begin{equation}\label{eq:end-step-1-dirichlet-prop}\EE\sqa{\Wb^2(\mu_{n}, 1)} \geq \EE\sqa{  \Wb^2 \bra{\mu^t_n, 1}} - C \frac{  \sqrt{\log n \log \log n}}{n}.\end{equation}
Thus, it is sufficient to prove \eqref{eq:lower-bound-square} with $\mu^t_n$ instead of $\mu_n$.

\emph{Step 2 (PDE and energy estimates)}. Let $f^t_n$ be the solution to $-\Delta f^t_n = \mu^t_n -1$ with null Neumann boundary conditions, i.e., since $P_s \mu^t_n = \mu^{n,t+s}$,
\[ f^t_n = \int_t^\infty (\mu^s_n - 1 )ds.\]
  Recall that, as proved in \cite{AST} and \cite[Lemma 3.14]{AmGlau},
\begin{equation}\label{eq:A} \EE\sqa{ \int_{(0,1)^2} \abs{ \nabla f^t_n}^2 } = \frac{\log n}{ 4 \pi n} + O\bra{ \frac 1 n}\end{equation}
and 
\begin{equation}\label{eq:B} \EE\sqa{ \int_{(0,1)^2} \abs{ \nabla f^t_n}^4} \lesssim \bra{ \frac{\log n}{n}}^2.\end{equation}
In addition to these gradient estimates, we will also need the upper bounds
\begin{equation}\label{eq:C} \EE\sqa{ \int_{(0,1)^2} \abs{ f^t_n}^2 }\lesssim  \frac{ 1}{n},\quad\text{and} \quad  \EE\sqa{ \int_{(0,1)^2} \abs{ f^t_n}^4 }\lesssim  \frac{ 1}{n^2},\end{equation} 
which can be proved in a similar way. Indeed, for every $x\in (0,1)^2$,
\[ \mu^s_n(x) -1 = \frac 1 n \sum_{i=1}^n \bra{ p_s(x,X_i) - 1} = \frac 1 n\sum_{i=1}^n \xi_i\]
is a sum of centered i.i.d.\ random variables $(\xi_i)_{i=1}^n$. For $s \in (0,1)$, we bound from above
\[ \EE\sqa{ \xi^2_i } = \int_{(0,1)^2} p_s^2(x,y)  - 1 = p_{2s}(x,x) -1 \les s^{-1} \]
where we used the ultracontractivity \eqref{eq:UC} for the heat kernel  and similarly
\[ \EE\sqa{ \xi^4_i} \les  \int_{(0,1)^2} p_s^4(x,y) + 1 \les s^{-2} \int_{(0,1)^2} p_s^2(x,y)  + 1 \les s^{-3}.\]
This yields immediately that
\[ \int_{(0,1)^2}  \EE\sqa{ \bra{ \mu^s_n(x) -1}^2 } \les \frac 1 {ns}.\]
Using Rosenthal's inequality \cite{rosenthal1970subspaces} (see also \cite{Le17}) we also get
\[ \int_{(0,1)^2} \EE\sqa{  (\mu_n^s(x) - 1 )^4  }  \les  \frac{1}{n^3 s^{3}}  +   \frac 1 {n^2s^{2}}.\]
For $s>1$, we use \eqref{eq:stronger-ultra}, and $p\in \cur{2,4}$,
\begin{equation*}\label{eq:bound-norm-s-1}\begin{split} \nor{ \mu^s_n - 1}_p  & \le  \nor{ \mu^s_n - 1}_\infty \les e^{-cs/2} s^{-d/2} \nor{  \mu^{s/2}_n - 1}_1  \le e^{-cs/2} s^{-d/2} \nor{  \mu^{s/2}_n - 1}_p \\
& \les   e^{-cs/2}  \|  \mu^{1/2}_n - 1\|_p.
\end{split}\end{equation*}
Taking expectation and using also the upper bounds above for $s=1/2$, it follows that for some constant $c>0$,
\[ \int_{(0,1)^2} \EE\sqa{ \bra{ \mu^s_n(x) -1}^p } \les \frac{e^{-cs }}{n^{p/2}}.\]
For $p=2$, we obtain
\[  \int_t^\infty \EE\sqa{ \bra{ \mu^s_n(x) -1}^2 }^{1/2} \les  \int_t^1  \frac 1 {\sqrt{ns}} ds + \int_1^\infty \frac{e^{-cs/2 }}{\sqrt{n}} ds \les \frac{1}{\sqrt{n}},\]
and, for $p=4$,
\[\int_t^\infty \EE\sqa{ \bra{ \mu^s_n(x) -1}^4 }^{1/4} \les \int_t^1  \bra{ \frac{1}{n^3 s^{3}}  +   \frac 1 {n^2s^{2}}}^{1/4} ds + \int_1^\infty \frac{e^{-cs/4 }}{\sqrt{n}} ds \les \frac{1}{\sqrt{n}}.\]

We conclude, for $p=2$ that
\[\begin{split} \EE\sqa{ \int_{(0,1)^2} |f^t_n|^2} & = \int_t^\infty \int_t^\infty  \EE\sqa{ \int_{(0,1)^2} (\mu^{n,s_1} - 1)(\mu^{n,s_2} - 1)} d s_1 d s_2\\
& \le \int_t^\infty \int_t^\infty \EE\sqa{ \int_{(0,1)^2} \bra{ \mu^{n,s_1}(x) -1}^2 }^{1/2}\EE\sqa{\int_{(0,1)^2} \bra{ \mu^{n,s_2}(x) -1}^2 }^{1/2} d s_1 d s_2,\\
& \le \bra{ \int_t^\infty \EE\sqa{ \int_{(0,1)^2} \bra{ \mu^s_n(x) -1}^2 }^{1/2} d s}^2 \les \frac 1 n,
\end{split}\]
and similarly for $p=4$,
\[\begin{split} \EE\sqa{ \int_{(0,1)^2} |f^t_n|^4}  \le \bra{ \int_t^\infty \EE\sqa{ \int_{(0,1)^2} \bra{ \mu^s_n(x) -1}^4 }^{1/4} d s}^4 \les \frac 1 {n^2}.
\end{split}\]


\emph{Step 3 (Dual potential)}. For $\eta\in (0,1)$ we introduce a smooth cut-off function $\chi_{\eta}: (0,1)^2 \to [0, 1]$ such that $\chi_{\eta}(x)=1$ on $[\eta, 1-\eta]^2$, $\chi_{\eta}(x) = 0$ on $(0,1)^2 \setminus [\eta/2, 1-\eta/2]^2$, with 
\[ \nor{ \nabla \chi_\eta}_\infty \lesssim \eta^{-1}, \quad \nor{ \nabla^2 \chi_ \eta}_\infty \lesssim \eta^{-2}.\]
The function $f_{n} = f^t_n \chi_{\eta}$ can be extended by periodicity to a function on the  flat torus $\mathbb{T}^2$. By \cite[Corollary 3.3]{AST} applied to $\mathbb{T}^2$, there exists $g_n\in C(\mathbb{T}^2)$  such that for $x$, $y \in \mathbb{T}^2$,
\[  f_n (x) +  g_n (y) \le \frac { d_{\mathbb{T}^2}(x,y)^2}{2}, \quad \text{and}\quad \int_{\mathbb{T}^2}( f_n + g_n)  \ge - e^{\nor{\Delta  f_n}_\infty} \int_{\mathbb{T}^2} |\nabla  f_n|^2.\]
We may naturally interpret $g_n$ as being defined on $(0,1)^2$, so that using $d_{\mathbb{T}^2}(x,y) \le |x-y|^2$, we have for $x$, $y \in (0,1)^2$,
\[  f_n (x) +  g_n(y) \le \frac{|x-y|^2}{2}.\]
To prove that $g_n(x) = 0$ on $\partial (0,1)^2$, we recall that \cite[Remark 3.4]{AST} $g_n$ is also the unique viscosity solution to the Hamilton-Jacobi equation $\partial_t u +\frac 1 2 |\nabla u|^2 = 0$ with initial condition $- f^n$, hence, it coincides with the Hopf-Lax solution
\[  g_n(x) = \inf_{y \in (0,1)^2} \cur{  - f_n(y) + \frac{ d_{\mathbb{T}^2}(x,y)^2}{2}}.\]
From this representation, we see at once that, if
 \begin{equation}\label{eq:norm-f-hj} \nor{ f_n}_\infty \le \frac{\eta^2}{8}\end{equation}
 then $ g_n(x) = 0$ on $\partial (0,1)^2$, because that $f_n(y) = f^t_n(y) \chi_\eta(y) = 0$ if $x \in \partial(0,1)^2$ and $d_{\mathbb{T}^2}(x,y)<\eta/2$.

Let us also notice that, since $\nabla  f_n = \nabla f^t_n \chi_\eta + f^t_n \nabla \chi_\eta$, we have by Young and Cauchy-Schwarz
\[ \begin{split} &  \int_{(0,1)^2} \abs{ \nabla  f_n - \nabla f^t_n}^2   \les  \int_{(0,1)^2} |\nabla f^t_n|^2 (1- \chi_\eta)^2 + \int_{(0,1)^2} \abs{f^t_n \nabla \chi_\eta}^2\\
& \quad \les \bra{ \int_{(0,1)^2} |\nabla f^t_n|^4}^{1/2} \eta^{1/2}  + \int_{(0,1)^2} |f^t_n|^2  \eta^{-2},\end{split}\]
which in expectation gives by \eqref{eq:B} and \eqref{eq:C}
\begin{equation}\label{eq:bound-exp-nabla-nabla} \EE\sqa{ \int_{(0,1)^2} \abs{ \nabla  f_n - \nabla f^t_n}^2 } \les \frac{\log n}{n} \eta^{1/2}  + \frac{1}{n} \eta^{-2}.\end{equation}
Using \eqref{eq:A} and Young inequality we find for every $\eps>0$,
\[
 \limsup_{n \to \infty} \lt|\frac {n}{\log n} \EE\sqa{ \int_{(0,1)^2} \abs{\nabla  f_n}^2 }-  \frac{1}{4 \pi}\rt|\les  \eps + \frac{C}{\eps} \eta^{1/2}.
\]
Choosing $\eps=\eta^{1/4}$ we conclude 
\begin{equation}\label{eq:gradfn2}
 \limsup_{n \to \infty} \lt|\frac {n}{\log n} \EE\sqa{ \int_{(0,1)^2} \abs{\nabla  f_n}^2 }-  \frac{1}{4 \pi}\rt|\les \eta^{1/4}.
\end{equation}
Arguing similarly, we have
\begin{equation}\label{eq:gradfn4}  \bra{ \frac {n}{\log n}}^2 \EE\sqa{ \int \abs{\nabla  f_n}^4} \les 1 + \frac{ 1}{\eta^4 (\log n)^2}.\end{equation}
Let us finally argue that 
\begin{equation}\label{eq:infinity}
 \nor{f^t_n}_\infty+\nor{\nabla f^t_n}_\infty+\nor{\Delta f^t_n}_\infty \les \nor{ \nabla^2 f^t_n}_\infty.
\end{equation}
Since the estimate $\nor{\Delta f^t_n}_\infty \les \nor{ \nabla^2 f^t_n}_\infty$ is clear and $\nor{f^t_n}_\infty\les \nor{\nabla f^t_n}_\infty$ follows from $\int_{(0,1)^2} f^t_n=0$, 
we are left with the proof of $\nor{\nabla f^t_n}_\infty\les  \nor{ \nabla^2 f^t_n}_\infty$. Fix $x_2\in(0,1)$ and consider the function 
$\phi(x_1)=f^t_n(x_1,x_2)$. We have $\phi'(x_1)=\nabla f^t_n(x_1,x_2)\cdot e_1$ (here $(e_1,e_2)$ is the canonical basis of $\R^2$). By the Neumann boundary conditions, $\phi'(0)=\phi'(1)=0$, 
so that $\nor{\phi'}_\infty\le \nor{\phi''}_\infty\les \nor{\nabla^2 f^t_n}_\infty$ and thus $\nor{\nabla f^t_n\cdot e_1}_\infty\les \nor{\nabla^2 f^t_n}_\infty$. Exchanging the roles of $x_1$ and $x_2$ concludes the proof of the claim.\\

\emph{Step 4 (Conclusion)} In the event $A_{n} = \cur{ \nor{ \nabla^2 f^t_n}_\infty < 1/\log n}$  we use $(f_n,g_n)$, as defined in the previous step, in the duality \eqref{eq:duality}. The event $A_n$ has large probability, i.e. 
\begin{equation}\label{eq:highprob} \PP\bra{  A_{n}^c} \lesssim n^{-k},\end{equation}
for every $k >0$ (with an implicit constant depending on $k$ only, see \cite{AmGlau}), hence it is sufficient to bound from below the expectation of $\Wb^2(\mu^t_n, 1)$ on $A_n$ 
(using the trivial  bound $\Wb^2(\mu^t_n, 1) \le 2$ on $(A^n)^c$). 
If $A_n$ occurs, we  have by \eqref{eq:infinity},
\[ \nor{f^t_n}_\infty+\nor{\nabla f^t_n}_\infty+\nor{\Delta f^t_n}_\infty \les  \frac{1}{\log n},\]
%
so that \eqref{eq:norm-f-hj} holds if $n$ is sufficiently large (for fixed $\eta$) and
\[  \nor{ \Delta  f_n}_\infty \les \nor{\nabla^2 (f^t_n \chi_\eta) }_{\infty} \les \frac{1}{ \eta^{2} \log n }  = \omega_\eta(n),\]
where for fixed $\eta$, $\lim_{n\to \infty} \omega_\eta(n)=0$.   Thus,
\[\begin{split}
\frac12 \Wb^2(\mu^t_n, 1) & \geq
	\int_{(0,1)^2} ( f_n+ g_n) + \int  f_n (\mu^t_n-1) \\
&\geq - e^{\omega_\eta(n)} \int_{(0,1)^2} \frac{\abs{\nabla  f_n }^2}2 
	- \int_{(0,1)^2}  f_n \Delta  f^t_n  \\
&= \left(1-\frac{e^{\omega_\eta(n)}}{2}\right)
 \int_{(0,1)^2} \abs{\nabla f_n}^2 + \int_{(0,1)^2} \nabla f_n \cdot \bra{ \nabla   f^t_n - \nabla f_n }.
\end{split}\]
Taking expectations, we bound the second term in the right-hand side using \eqref{eq:bound-exp-nabla-nabla} and \eqref{eq:gradfn2},
\[\limsup_{n \to \infty} \frac{n}{\log n} \EE\sqa{ \int_{(0,1)^2}\abs{ \nabla f_n\cdot   \bra{ \nabla   f^t_n - \nabla f_n }}} \les \eta^{1/4}.\]
For the first term, we have
\[\begin{split}
 & \liminf_{n\to\infty} \frac{n}{\log n} \EE\left[
	(2- e^{\omega_\eta(n)})\chi_{A_n} \int_{(0,1)^2} \abs{\nabla f_{n}}^2  \right]
	 \\
& \geq \liminf_{n\to\infty}  \frac{n}{\log n} \EE\left[
	\int_{(0,1)^2} \abs{\nabla f_{n}}^2  \right]
	- \limsup_{n\to\infty} \frac{n}{\log n}  \EE\left[
	\chi_{A_n^c} \int_{(0,1)^2} \abs{\nabla f_{n}}^2  \right]
	 \\
&\stackrel{\eqref{eq:gradfn2}}{=} \frac 1{4\pi} - C\eta^{1/4},
\end{split}\]
where we used in the last line that 
\begin{multline*}
\limsup_{n\to\infty} \frac{n}{\log n}\EE\left[
	\chi_{A_n^c} \int_{(0,1)^2} \abs{\nabla f_{n}}^2  \right]  \leq \limsup_{n\to\infty} \PP(A_n^c)^{1/2}
\frac{n}{\log n}\EE\left[ \int_{(0,1)^2} \abs{\nabla f_{n}}^4  \right]^{1/2}\\
\stackrel{\eqref{eq:highprob}\&\eqref{eq:gradfn4}}{=} 0.
\end{multline*}
Letting $\eta \to 0$ we finally obtain the thesis.
\end{proof}
\end{proposition}


Next, we deal with the analogue result for the bipartite matching. In fact, we even consider the case of different number of points. 
Let us point out that we only prove a lower bound since this is what we will later use in the proof of Theorem \ref{thm:main-ref-bip}. 
The corresponding upper bound may be obtained as a corollary of that theorem or more elementarily by arguing as in \eqref{eq:upper-bound-proof-bip}.

\begin{proposition}\label{prop:astbip-dirichlet}
 Let $(X_i,Y_i)_{i=1}^\infty$ be i.i.d uniformly distributed in $(0,1)^2$, then there
 exists a rate function\footnote{recall that a rate function is a generic decreasing function $\omega$ such that $\lim_{t\to\infty} \omega(t)=0$.} $\omega$ such that for every $n\le m$,
\[
  \EE\sqa{ \Wb_{(0,1)^2, 2}^2 \bra{\frac 1 n \sum_{i=1}^n \delta_{X_i}, \frac 1 m \sum_{i=1}^m \delta_{Y_i} }} \ge  \frac{\log n}{4 \pi n}\lt(1+ \frac{n}{m} - \omega(n)\rt).
 \]

\end{proposition}
\begin{proof}
The proof is very similar to the lower bound in the previous Proposition \ref{prop:ast-dirichlet} so we only sketch it.  Define $\mu_n=\frac{1}{n}\sum_{i=1}^n \delta_{X_i}$ and $\lambda_m=\frac{1}{m}\sum_{i=1}^m \delta_{Y_i}$, and write $W$ for $W_{(0,1)^2,2}$ and similarly for $\Wb$. 
  We denote by $\mu^t_n$ and $\lambda^t_m$ the Neumann heat-kernel regularizations of $\mu_n$ and $\lambda_m$. We define $f^t_n$ as the solution to $-\Delta f^t_n=\mu^t_n-1$ and $f^t_m$ the solution of $-\Delta f^t_m= \lambda^t_m-1$. We finally set $f_{n,m}^t=f^t_n-f^t_m$, and choose $t= (\log n)^4/n$. Notice that if $t_m=(\log m)^4/m$ then $t_m\le t$.
 
 Repeating \textit{Step 1} of the proof of Proposition \ref{prop:ast-dirichlet} we have, by the  triangle inequality and $\Wb\le W$, 
  \begin{align*}
 \Wb^2& \bra{\mu^t_n,\lambda^t_m} - \Wb^2(\mu_n,\lambda_m) \\
   &\le\bra{W(\mu^t_n,\mu_n)+\Wb(\mu_n,\lambda_m)+W(\lambda^t_m,\lambda_m)}^2- \Wb^2(\mu_n,\lambda_m)\\
   &\les W^2(\mu^t_n,\mu_n)+W^2(\lambda^t_m,\lambda_m)+W(\mu_n,\lambda_m)\bra{W(\mu^t_n,\mu_n)+W(\lambda^t_m,\lambda_m)}.
  \end{align*}
The refined contractivity estimate \cite[Theorem 5.2]{AmGlau} gives
\[
 \EE\lt[W^2(\mu^t_n,\mu_n)\rt]\les  \frac{\log \log n}{n},
\]
and moreover, since $t=(\log n)^4/n=\gamma/m$ with $\gamma=(\log n)^4(m/n)\gg \log m$, applying again \cite[Theorem 5.2]{AmGlau} we have 
\[
 \EE\lt[W^2(\lambda^t_m,\lambda_m)\rt]\les \frac{\log \gamma}{m}\les \frac{\log \log n + \log \lt(\frac{m}{n}\rt)}{m}\les \frac{\log \log n}{n},
\]
where we used that $m\ge n$ and thus $\frac{n}{m} \log \lt(\frac{m}{n}\rt)$ is bounded. Using also that 
\[
 \EE\lt[W^2(\mu_n,\lambda_m)\rt]\les \EE\lt[W^2(\mu_n,1)\rt]+\EE\lt[W^2(\lambda_m,1)\rt]\les \frac{\log n}{n} + \frac{\log m}{m}\les \frac{\log n}{n}
\]
we find the analogously to \eqref{eq:end-step-1-dirichlet-prop} that  for some constant $C>0$,
\[
 \EE\lt[\Wb^2(\mu_n,\lambda_m)\rt]\ge \EE\lt[\Wb^2(\mu^t_n,\lambda^t_m)\rt] - C \frac{\sqrt{\log n \log \log n}}{n}.
\]
Turning now to \textit{Steps 2-4}, we see that the proof can be repeated verbatim using $f_{n,m}^t$ instead of $f^t_n$ and the triangle inequality to obtain the various estimates from the corresponding ones on $f^t_n$ and $f^t_m$. 
The only additional ingredient is that  using  integration  by parts and independence, we have  the following orthogonality property
\[\begin{split}
\EE\lt[  \int_{(0,1)^2} \nabla f^t_n\cdot \nabla f^t_m\rt] &= - \EE\sqa{ \int_{(0,1)^2}  f^t_n \Delta f^t_m} = -  \int_{(0,1)^2} \EE\sqa{ f^t_n} \EE\sqa{ \lambda^t_m}\\
& = 0.
\end{split} \]
 Therefore, appealing once again to \cite[Lemma 3.14]{AmGlau},
\begin{align*}
 \EE\lt[\int_{(0,1)^2} |\nabla f_{n,m}^{t}|^2\rt]&= \EE\lt[\int_{(0,1)^2} |\nabla f^t_n|^2\rt]+\EE\lt[\int_{(0,1)^2} |\nabla f^t_m|^2\rt]+ 2\EE\lt[\int_{(0,1)^2} \nabla f^t_n\cdot \nabla f^t_m\rt]\\
 &=\EE\lt[\int_{(0,1)^2} |\nabla f^t_n|^2\rt]+\EE\lt[\int_{(0,1)^2} |\nabla f^t_m|^2\rt]\\
 &= \frac{|\log t|}{4\pi n}+\frac{|\log t|}{4\pi m} + O\lt(\frac{1 }{n}\rt)\\
 &=\frac{\log n}{4\pi n}\lt( 1+\frac{n}{m}+ \omega(n)\rt). \qedhere
\end{align*}

\end{proof}

\begin{remark}
Arguing as in \cite[Theorem 5.2]{AST}, we can also provide a lower bound for the Wasserstein distance of order $1$:
\[ \liminf_{n \to \infty} \frac{1}{\sqrt{\log n}} \EE\sqa{ \Wb_{(0,1)^2, 1} \bra{\frac 1 n \sum_{i=1}^n \delta_{X_i}, \frac 1 n \sum_{i=1}^n \delta_{Y_i}}} >0.\]
The proof is exactly as in \cite[Theorem 5.2]{AST}. The only difference lies in the choice of the Lipschitz function  in the dual  description of the Wasserstein distance. The function $\phi$ from \cite[Theorem 5.2]{AST} must be replaced by $\phi \chi_\eta$ where $\chi_\eta$ is defined  in the proof above. 
\end{remark}

\section{Proof of Theorem~\ref{thm:main-ref}}\label{sec:main-ref}

 We recall that for a given Lipschitz and connected domain $\Omega$ and a H\"older continuous and uniformly strictly positive probability density $\rho$ on $\Omega$, we consider
 $(X_i)_{i=1}^\infty$ i.i.d. random variables with common distribution $\rho$. Letting 
 \[
  \mu_n=\frac{1}{n}\sum_{i=1}^n \delta_{X_i},
 \]
 we want to prove that 
 \begin{equation}\label{eq:toproveTh1}
  \lim_{n\to\infty} \frac{n}{\log n}\EE\lt[W_2^2\lt(\mu_n,\rho\rt)\rt]=\frac{|\Omega|}{4\pi}.
 \end{equation}
Without loss of generality, we assume that $|\Omega| = 1$.  We also omit to specify $p=2$ and simply write $W_{Q} = W_{Q,2}$ and $\Wb_Q = \Wb_{Q,2}$ for $Q \subseteq \Omega$. 
Let us introduce some further notation: for $r>0$ we consider $(X_i^r)_{i=1}^\infty$, i.i.d.\ uniformly distributed random variables in $[0,r]^2$. We then define 
the average cost functions 
\[
 F_r(n)= \EE\sqa{ W^2_{[0,r]^2}\bra{ \frac 1 n \sum_{i=1}^n \delta_{X^r_i}, \frac 1 {r^2}}}, \quad Fb_r(n) = \EE\sqa{ \Wb^2_{[0,r]^2}\bra{ \frac 1 n \sum_{i=1}^n \delta_{X^r_i}, \frac 1 {r^2}}}
\]
By scaling, \eqref{eq:ast-cube-ref} and Proposition~\ref{prop:ast-dirichlet},  we have 
\begin{equation}\label{Fr}
 F_r(n)= r^2 F_1(n)\le r^2 \frac{\log n}{4\pi n}\lt(1+\omega(n)\rt), \quad \text{and} \quad Fb_r(n)= r^2 Fb_1(n)\ge r^2 \frac{\log n}{4\pi n}\lt(1-\omega(n)\rt).
\end{equation}
\emph{Step 1 (Whitney decomposition).} 
 We consider $\cur{Q_k}_k$ a Whitney decomposition of $\Omega$, see  \cite[Chapter 6]{stein2016singular}. For $\delta=\delta(n)>0$ to be 
 chosen below we set  $U_\delta=\{ Q_k \ : \ \diam(Q_k)\ge \delta\}$ and $V_\delta=\{Q_k \ : \ \diam(Q_k)<\delta\}$. For $r>0$ dyadic and small enough so that  Lemma~\ref{lem:heat-cube} applies with $r \diam(\Omega)$, we split each $Q_k$ into $r^{-2}$ sub-cubes.
 We let $V_\delta^r$ be the union of these cubes. Up to relabeling we have $U_\delta\cup V_\delta^r=\{\Omega_k\}_k$.  We finally define $\kappa_k= \mu_n(\Omega_k)/\rho(\Omega_k)$.
 \\
 For $\eps>0$, we estimate by the subadditivity inequality \eqref{eq:sub}
 \begin{equation}\label{sublog}
  \EE\lt[W^2_{\Omega}(\mu_n,\rho)\rt]\le (1+\eps) \sum_{k} \EE\lt[ W^2_{\Omega_k}(\mu_n, \kappa_k \rho)\rt] + \frac{C}{\eps} \EE\lt[W^2\lt(\sum_{k}  \kappa_k \rho \chi_{\Omega_k},\rho\rt)\rt]
 \end{equation}
 and the superadditivity inequality \eqref{eq:super-proof}
 \begin{equation}\label{suplog}
 \EE\lt[\Wb^2_{\Omega}(\mu_n,\rho)\rt]\ge  (1-\eps)\sum_{k}\EE\lt[ \Wb^2_{\Omega_k}(\mu_n,  \kappa_k \rho)\rt] - 
\frac{C}{\eps} \EE\lt[ \Wb^2_{\Omega}\lt(\sum_{k}  \kappa_k \rho \chi_{\Omega_k},\rho\rt)\rt].
\end{equation}
In both expressions,  we recognize in the right-hand sides a sum of ``local'' contributions and an additional  ``global'' term. We consider these contributions separately in the next two steps and  show in particular that the global term does not contribute in the limit.\\

\emph{Step 2 (Local term).} 
We let $N_{k} = n \mu_n(\Omega_k)$ be the number of points $X_i$ which belong to the cube $\Omega_k$. This is a random variable with binomial law and parameters $(n, \rho(\Omega_k))$.  
We decompose the sum in the right-hand side of \eqref{sublog} as
\[
 \sum_{k}  \EE\lt[ W^2_{\Omega_k}(\mu_n, \kappa_k \rho)\rt]
  = \sum_{U_\delta} \EE\lt[ W^2_{\Omega_k}(\mu_n, \kappa_k \rho)\rt]+\sum_{V_\delta^r} \EE\lt[ W^2_{\Omega_k}(\mu_n, \kappa_k\rho)\rt].
 \]
For the first term we use the naive estimate
\[  \EE\lt[ W^2_{\Omega_k}(\mu_n, \kappa_k \rho)\rt]\les \diam^2(\Omega_k) \EE\sqa{\mu_n(\Omega_k)}\le \delta^2 \rho(\Omega_k),\]
 to get
\begin{equation}\label{eq:NBr1}
  \sum_{U_\delta} \EE\lt[ W^2_{\Omega_k}(\mu_n, \kappa_k \rho)\rt]\les  
 \delta^2 \sum_{U_\delta}  \rho(\Omega_k) \les \delta^2\abs{ d(\cdot, \partial \Omega)\les \delta } \les \delta^3.
\end{equation}
As for the second term, we use \eqref{eq:holder-cube} from Lemma~\ref{lem:heat-cube} to infer that for every $\Omega_k\in V_\delta^r$,
\[
\begin{split}
 \EE\lt[ W^2_{\Omega_k}(\mu_n, \kappa_k \rho)\rt]& = \EE\lt[   \EE\lt[  W^2_{\Omega_k}(\mu_n, \kappa_k \rho) | N_{k} \rt]\rt]\\
 & \le (1+C r^\alpha)\EE\lt[\frac{N_{k}}{n} F_{\sqrt{|\Omega_k|}}(N_{k})\rt]\\
 &\stackrel{\eqref{Fr}}{\le}(1+Cr^\alpha) \frac{|\Omega_k|}{ 4\pi n} \EE\lt[\log(N_{k})(1+\omega(N_{k}))\rt].
\end{split}
\]
Using the concentration properties of the binomial random variable $N_k$ and the fact that  $\delta = \delta(n)$ satisfies $\lim_{n \to \infty} n \delta^2 =\infty$, we see that 
\begin{equation}\label{eq:FNk}
 \EE\lt[\log(N_{k})(1+\omega(N_{k}))\rt]= \log (n\rho(\Omega_k)) (1+\omega( n r^2\delta^2))
\end{equation}
 and we find
\begin{equation}\label{eq:NBr}
  \EE\lt[ W^2_{\Omega_k}(\mu_n, \kappa_k \rho)\rt]\le (1+C r^\alpha) \frac{\log (n\rho(\Omega_k))}{ 4\pi n} |\Omega_k| (1+\omega(nr^2\delta^2)). 
\end{equation}
We now claim that 
\begin{equation}\label{claim:sumOm}
 \sum_{k} |\Omega_k| |\log |\Omega_k||\les |\log r|.
\end{equation}
Indeed, 
 \begin{multline*}
  \sum_{k} |\Omega_k| |\log |\Omega_k||= \sum_{U_\delta} |\Omega_k| |\log |\Omega_k||+\sum_{V_{\delta}^r} |\Omega_k| |\log |\Omega_k||\\
  =\sum_{U_\delta} |Q_k| |\log |Q_k||+\sum_{V_{\delta}} |Q_k| |\log |Q_k| +\log r|\\
  \les \sum_k  |Q_k| |\log |Q_k|| +\sum_{V_\delta} |Q_k| |\log r|\les 1+|\log r|\les |\log r|,
 \end{multline*}
where we used  that  (notice that the finiteness of the integral below is ensured, for instance,
by the finiteness of the Minkowski content of $\Omega$, in turn ensured by Lipschitz regularity)
\[ \sum_{k} |Q_k|  |\log |Q_k|| \les \int_{\Omega} |\log d( \cdot , \partial \Omega)|  < \infty.\]
Thus, from \eqref{eq:NBr} and \eqref{claim:sumOm},
\begin{equation}\label{firstterm}
\sum_{k} \EE\lt[ W^2_{\Omega_k}(\mu_n, \kappa_k \rho)\rt]\le  (1+C r^\alpha) \frac{\log n}{ n}  \frac{1}{4\pi} + C\lt( \frac{|\log r|}{n}+\frac{\log n}{ n}\omega(n r^2\delta^2) +  \delta^3\rt).
\end{equation}
For the analogue term in \eqref{suplog}, we simply discard the terms with $\Omega_k\in U_\delta$ and for the remaining ones use again \eqref{eq:holder-cube} 
from Lemma~\ref{lem:heat-cube} and the function $Fb_r$ instead of $F_r$ to obtain 
 \begin{equation*}\label{firsttermb}
\sum_{k} \EE\lt[ \Wb^2_{\Omega_k}(\mu_n, \kappa_k \rho)\rt]\ge  (1-C r^\alpha) \frac{\log n}{ n}  \frac{1- C\delta}{4\pi} - C\lt( \frac{|\log r|}{n} +\frac{\log n}{ n}\omega (n r^2 \delta^2)\rt).
\end{equation*}

\emph{Step 3 (Global term).} We now turn to the second term in the right-hand side of \eqref{sublog}, for which we show that
\begin{equation}\label{eq:global} 
\EE\sqa{ W^2_{\Omega}\bra{\sum_{k}  \kappa_k \rho \chi_{\Omega_k},\rho} } \les \frac{ |\log r|}{n}.
\end{equation}
Notice that the inequality $\Wb \le W$ gives an inequality also for the corresponding term in  \eqref{suplog}.\\
We first use Lemma~\ref{lem:peyre} to deduce 
\begin{multline*}
 W^2_{\Omega}\bra{\sum_{k}  \kappa_k \rho \chi_{\Omega_k},\rho}  \les \lt\|\sum_{k}  (\kappa_k-1) \rho \chi_{\Omega_k}\rt\|_{H^{-1}}^2 \\
 =\lt\|\sum_{k}  (\kappa_k-1) \lt(\rho \chi_{\Omega_k}-\rho(\Omega_k)\chi_{\Omega}\rt)\rt\|_{H^{-1}}^2 =\lt\|\sum_{k}  (\kappa_k-1)  f_{k} \rt\|_{H^{-1}}^2,
\end{multline*}
where 
\[ f_k = \rho \chi_{\Omega_k}-\rho(\Omega_k)\chi_{\Omega}.\]
 Taking expectation and expanding the squares we have 
\[
     \lt\|\sum_{k}  (\kappa_k-1)  f_{k} \rt\|_{H^{-1}}^2
    = \sum_{k} \EE[(\kappa_k-1)^2] \|f_k\|_{H^{-1}}^2 + \sum_{j\neq k} \EE[ (\kappa_j-1)(\kappa_k-1)]\langle f_j,f_k\rangle_{H^{-1}}.
\]
For the first sum, we use that
\[ \EE[(\kappa_k-1)^2] = \frac{1}{\rho^2(\Omega_k)} \EE\sqa{ \bra{ \mu_n(\Omega_k) - \rho(\Omega_k)}^2} \le \frac{1}{n \rho(\Omega_k)} \les \frac{1}{n|\Omega_k|}\]
 and, by Lemma~\ref{lem:Sob} with $f=f_k$, since $\nor{ f_k}_\infty \les 1$  and $\nor{ f_k}_1 \sim \rho(\Omega_k)$,
\begin{equation}\label{eq:Sobinter}
 \nor{f_k}_{H^{-1}}^2 \les \rho(\Omega_k)^2 |\log \rho( \Omega_k)| \les  |\Omega_k|^2 \abs{ \log |\Omega_k| }
\end{equation}
to get
\[
 \sum_{k} \EE[(\kappa_k-1)^2] \|f_k\|_{H^{-1}}^2 \les 
 \frac{1}{n} \sum_{k} |\Omega_k| |\log |\Omega_k||
 \les \frac{1}{n} \sum_{k} |\Omega_k| |\log |\Omega_k||\stackrel{\eqref{claim:sumOm}}{\les}\frac{|\log r|}{n}.\]

For the second sum, we use that if $j\neq k$,  
\[
 \EE[ (\kappa_j-1)(\kappa_k-1)]= \frac{\EE\sqa{\mu_n(\Omega_j)\mu_n(\Omega_k)}}{\rho(\Omega_j)\rho(\Omega_k)}-1= -\frac{1}{n}
\]
together with \eqref{eq:Sobinter} and Cauchy-Schwarz inequality to conclude 
\[
\begin{split}  \sum_{j\neq k}  \EE[ (\kappa_j-1)(\kappa_k-1)]\langle f_j,f_k\rangle_{H^{-1}}& \les \frac 1 n \sum_{j\neq k} \nor{ f_j}_{H^{-1}}  \nor{ f_k}_{H^{-1}} \\
&\les \frac{1}{n} \sum_{j\neq k}|\Omega_j| |\log |\Omega_j||^{1/2} |\Omega_k| |\log |\Omega_k||^{1/2}\\
&\les \frac{1}{n}\lt(\sum_{k}|\Omega_k| |\log |\Omega_k||^{1/2}\rt)^2\\
&\les \frac{|\log r|}{n},
\end{split}
\]
where we argued as above to bound $\sum_{k}|\Omega_k| |\log |\Omega_k||^{1/2} \les |\log r|^{1/2}$. This proves \eqref{eq:global}. \\

\emph{Step 4 (Conclusion).} Putting together \eqref{firstterm} and \eqref{eq:global}, we conclude that
\[
 \EE[W^2(\mu_n,\rho)]\le (1+\eps)(1+C r^\alpha) \frac{\log n}{ n} \frac{1}{4\pi} + C \lt(\frac{\log n}{n} \omega(n r^2\delta^2) +  \delta^3 +\frac{1}{\eps}  \frac{|\log r|}{ n}\rt)
\]
and similarly
\[
 \EE[\Wb^2(\mu_n,\rho)]\ge (1-\eps)(1-C r^\alpha) \frac{\log n}{ n} \frac{1-C\delta}{4\pi} - \lt( \frac{\log n}{n} \omega (n r^2 \delta^2) +\frac{1}{\eps}  \frac{|\log r|}{ n}\rt).
\]
We now choose $\delta = \delta(n) = n^{-\beta}$ with $\beta \in (\frac 1 3, \frac 1 2)$, so that the condition $\lim_{n\to \infty} n \delta^2  = \infty$ holds but also $\lim_{n\to \infty} n\delta^3/\log n  =0$. This yields
\begin{multline*}
 (1+\eps)(1+C r^\alpha) \frac{1}{4\pi}\ge  \limsup_{n\to \infty}\frac{n}{\log n}\EE[W^2(\mu_n,\rho)]\ge \liminf_{n\to \infty}\frac{n}{\log n}\EE[\Wb^2(\mu_n,\rho)]\\
 \ge (1-\eps)(1-C r^\alpha) \frac{1}{4\pi}.
\end{multline*}
Letting finally $r \to 0$ and $\eps\to 0$ we conclude the proof of \eqref{eq:toproveTh1}.

 \section{Proof of Theorem~\ref{thm:main-ref-bip}}\label{sec:main-bip}

For $(X_i,Y_i)_{i=1}^\infty$ i.i.d.\ random variables with common distribution $\rho$ in $\Omega$, we write 
\[
 \mu_n=\frac{1}{n}\sum_{i=1}^n \delta_{X_i} \qquad \textrm{and} \qquad \lambda_m =\frac{1}{m}\sum_{j=1}^m \delta_{Y_j}.
\]
We now prove that 
\begin{equation*}\label{eq:toproveTh2}
 \lim_{\substack{n,m \to \infty \\ m/n \to q}} \frac{n}{\log n} \mathbb{E} \sqa{ W_2^2\bra{ \mu_n, \lambda_m} } = \frac {|\Omega|} {4 \pi}\bra{ 1 + \frac{1}{q}}.
\end{equation*}
As above, we may assume by scaling that $|\Omega|=1$ and we will omit to specify $p=2$, simply writing $W_{Q} = W_{Q,2}$ and $\Wb_Q = \Wb_{Q,2}$ for $Q \subseteq \Omega$. 
For $n \le m$, we will mostly focus on  the lower bound, 
\begin{equation}\label{eq:lower-bound-proof-bip}  \EE\sqa{\Wb^2_\Omega(\mu_n, \lambda_m)}  \ge   \frac{\log n}{4\pi n } \lt(1+ \frac{n}{m} - \omega(n)\rt). \end{equation}
  Indeed, we first show how the upper bound,
\begin{equation}\label{eq:upper-bound-proof-bip} \EE\sqa{W^2_\Omega(\mu_n, \lambda_m)} \le  \frac{\log n}{4\pi n } \lt(1+ \frac{n}{m} +\omega(n)\rt),\end{equation}
can be then quickly obtained as a consequence of \cite[Proposition 4.9]{AST}\footnote{in fact the statement there wrongly contains an equality that should be instead an inequality $\le$}. Letting $T^{\mu_n}$
(respectively $T^{\lambda_m}$) be the optimal transport maps between $\rho$ and $\mu_n$ (respectively $\lambda_m$), by independence we have 
\begin{equation} \begin{split}\label{eq:doubling-trick} \EE\sqa{ W^2_\Omega \bra{\mu_n, \lambda_m }}&\le \EE\sqa{ \int_{\Omega}|T^{\mu_n}-T^{ \lambda_m}|^2\rho }\\
  &=  \EE\sqa{ W^2_\Omega \bra{\mu_n, \rho }}+ \EE\sqa{ W^2 _\Omega\bra{ \lambda_m,\rho }}\\
  & \quad- 2\int_{\Omega}\EE\sqa{(T^{\mu_n}-x)}\cdot \EE\sqa{(T^{\lambda_m}-x)}\rho.
\end{split} \end{equation}
 We start with the case $n=m$. In this case $\mu_n$ and  $\lambda_m$ have the same law, hence the last term above becomes
 \[ \int_{\Omega}\EE\sqa{(T^{\mu_n}-x)}\cdot \EE\sqa{(T^{\lambda_m}-x)}\rho= \int_{(0,1)^2}\abs{ \EE\sqa{(T^{\mu_n}-x)}}^2\rho.\]
By \eqref{eq:lower-bound-proof-bip}  and Theorem~\ref{thm:main-ref} (recall \eqref{eq:toproveTh1}) we get 
 \[ \limsup_{n \to \infty} \frac{n}{\log n}  \int_{\Omega}\abs{ \EE\sqa{(T^{\mu_n}-x)}}^2 \rho \le \lim_{n \to \infty}  \frac{n}{\log n}\lt(\EE\sqa{ W^2_\Omega \bra{\mu_n, \rho }} - \frac 1 2 \EE\sqa{ W^2 _\Omega\bra{\mu_n, \lambda_n }} \rt)= 0.\]
 Therefore, 
 \[ \EE\sqa{ W^2_\Omega \bra{\mu_n, \rho }} +  \int_{\Omega}\abs{ \EE\sqa{(T^{\mu_n}-x)}}^2\rho  \le  \frac{\log n}{4 \pi n} \bra{ 1 +\omega(n)}. \]
 We now turn to the general case $n \le  m$. Using  the inequality $-2ab \le a^2 + b^2$ for the last term in \eqref{eq:doubling-trick}, we obtain (here we use that $\omega(m)\le \omega(n)$ and $\log m= \log \frac{m}{n}+ \log n$)
\begin{align*}
  \EE\sqa{ W^2 _\Omega\bra{\mu_n, \lambda_m }}&\le \frac{\log n}{4\pi n } (1+\omega(n))+ \frac{\log m}{4\pi m } (1+\omega(m))\\
  &= \frac{\log n}{4\pi n } \lt(1+ \frac{n}{m} + \lt(\frac{\log m}{\log n}-1\rt)\frac{n}{m} + \frac{n}{m} \frac{\log m}{\log n} \omega(m)+ \omega(n)\rt)\\
  &\le \frac{\log n}{4\pi n } \lt(1+ \frac{n}{m} + \frac{1}{\log n} \frac{ \log \frac{m}{n}}{\frac{m}{n}}\lt(1+\omega(n)\rt) + \lt(\frac{n}{m} +1\rt)\omega(n)\rt)\\
  &\le \frac{\log n}{4\pi n } \lt(1+ \frac{n}{m} +\omega(n)\rt).
  \end{align*}
 This proves \eqref{eq:upper-bound-proof-bip}.\\
  
We thus focus  on the proof of \eqref{eq:lower-bound-proof-bip}. For this we follow the same steps as  in Theorem~\ref{thm:main-ref}.
However, we need to be more careful when estimating the ``global'' term. Indeed, we cannot apply directly Lemma~\ref{lem:peyre} since 
the measure $\lambda_m$ is singular. To fix this issue, we  first  slightly modify the Whitney decomposition to avoid cubes whose measure is too small. 
This guarantees that with overwhelming probability, every element of the partition contains many points so that  we may argue as in \cite[Proposition 5.2]{GolTre}.
Notice that of course we could have used Lemma \ref{lem:decomp} also in the proof of Theorem \ref{thm:main-ref}.

To simplify the exposition, we postpone the proof of the following geometric result to Appendix~\ref{app:proof-lemma-dec}.


\begin{lemma}\label{lem:decomp}
Let $\Omega\subset \R^d$ be a bounded open set with Lipschitz boundary, let $\{Q_k\}_k$ be a Whitney decomposition of $\Omega$. For every $\delta>0$
small enough (depending on $\Omega$), letting $V_\delta=\{Q_k \ : \ \diam(Q_k) \ge \delta\}$, there exists a family $U_\delta=\{\Omega_k\}_k$ of disjoint open sets such that
$\diam(\Omega_k)\les \delta$, $|\Omega_k|\sim \delta^d$ and $U_\delta\cup V_\delta$ is a partition of $\Omega$.
\end{lemma}

\emph{Step 1 (Whitney-type decomposition and reduction to a good event).} For  $\delta = \delta(n)>0$ to be specified below, let  $U_\delta$ and $V_\delta$ be given by Lemma \ref{lem:decomp}. For $r>0$ dyadic we divide each cube
$Q_k\in V_\delta$ in $r^{-2}$ equal sub-cubes  of sidelength $r\ell(Q_k)$. We let $V_\delta^r$ be their collection and relabel $U_\delta\cup V_\delta^r=\{\Omega_k\}_k$.
 We set  
\begin{equation}\label{choicetheta}
 \theta=\frac{1}{\sqrt{ \log  n}}
\end{equation}
and for every $k$ such that $\lambda_m(\Omega_k)\neq 0$
\[
 \kappa_k= \frac{\mu_n(\Omega_k)}{\lambda_m(\Omega_k)}.
\]
If instead $\lambda_m(\Omega_k)= 0$ we arbitrarily set $\kappa_k=0$. Define the event
\[
 A=\lt\{ \forall \Omega_k \in V_\delta^r\cup U_\delta, \ \lt|1-\frac{\mu_n(\Omega_k)}{\rho(\Omega_k)}\rt|+ \lt|1-\frac{\lambda_m(\Omega_k)}{\rho(\Omega_k)}\rt|\le \frac{ \theta}{4}\rt\}.
\]
Notice that if $|1-\frac{\mu_n(\Omega_k)}{\rho(\Omega_k)}|+ |1-\frac{\lambda_m(\Omega_k)}{\rho(\Omega_k)}|\le \frac{ \theta}{4}$ then 
\begin{equation}\label{claim:likely}
 |1-\kappa_k|\le \frac{\theta}{2}.
\end{equation}
We claim that if  
\[
 \delta= \frac{1}{n^\beta}, \quad \text{for some $\beta\in (0,1/2)$,} \quad  r \sim \frac{1}{\log(n)},
\]
then 
\begin{equation}\label{claim:smallprob}
 \PP( A^c) \les \exp\lt(-c \frac{n^{1-2\beta}}{(\log n)^3}\rt),
\end{equation}
so that we can restrict ourselves to the event $A$ in the following steps.
By  a union bound, to prove the claim, we bound 
\[
 \PP(A^c)
 \le \sum_k \PP\bra{ \lt|1-\frac{\mu_n(\Omega_k)}{\rho(\Omega_k)}\rt|\ge \theta} +\sum_k\PP\bra{ \lt|1-\frac{\lambda_m(\Omega_k)}{\rho(\Omega_k)}\rt|\ge \theta}.
\]
We focus only on the terms involving $\mu_n$ (those with $\lambda_m$ are analogous, recalling that $n \le m$). For every $k$, the number of points $n \mu_n(\Omega_k)$ is a binomial random variable with parameters $(n, \rho(\Omega_k))$.
Hence, by Chernoff bounds
\[\PP\bra{ \lt|1-\frac{\mu_n(\Omega_k)}{\rho(\Omega_k)}\rt|\ge \theta} \le \exp\lt(- \frac{n \rho(\Omega_k)}{2\log n}\rt).\]
Summing this over $k$ and using that by definition of a Whitney partition, for every $j\in \N$ such that $2^{-j}\ge \delta$,
\[\#\{ \Omega_k\in V^r_\delta \ : \  \ell(\Omega_k)\in r (2^{-(j+1)}, 2^{-j}]\}\les r^{-2} 2^j,\] we find
\begin{align*}
 \sum_k \PP\bra{ \lt|1-\frac{\mu_n(\Omega_k)}{\rho(\Omega_k)}\rt|\ge \theta}&\le \sum_{ V^r_\delta} \exp\lt(- \frac{n \rho(\Omega_k)}{2\log n}\rt)+\sum_{ U_\delta} \exp\lt(- \frac{n \rho(\Omega_k)}{2\log n}\rt)\\
 &\les \sum_{2^{-j}\ge \delta} r^{-2} 2^j \exp\lt(-c \frac{n r^2 2^{-2j}}{\log n}\rt) +\delta^{-1} \exp\lt(-c\frac{n \delta^2}{\log n}\rt)\\
 &\les r^{-2} \delta^{-1}\exp\lt(-c \frac{nr^2 \delta^2}{\log n}\rt)+ \delta^{-1}\exp\lt(-c\frac{n \delta^2}{\log n}\rt)\\
 &\les r^{-2} \delta^{-1}\exp\lt(-c \frac{nr^2 \delta^2}{\log n}\rt).
\end{align*}
Up to replacing the constant $c>0$ with a smaller one, to control the terms  $r^{-2}\delta^{-1} \les n^{\beta} (\log n)^2$, 
we obtain  \eqref{claim:smallprob}.
We end this step applying \eqref{eq:mainsup} and using $\Wb\le W$ to get for every $\eps>0$, 
\begin{multline*}
 \Wb^2_\Omega (\mu_n,\lambda_m)\ge (1-\eps)\sum_{k} \Wb^2_{\Omega_k}(\mu_n,\kappa_k \lambda_m)- \frac{C}{\eps} W_{\Omega}^2\lt(\sum_{k} \kappa_k \chi_{\Omega_k}\lambda_m,\lambda_m\rt)\\
 \ge  (1-\eps)\sum_{V_\delta^r} \Wb^2_{\Omega_k}(\mu_n,\kappa_k \lambda_m)- \frac{C}{\eps} W_{\Omega}^2\lt(\sum_{k} \kappa_k \chi_{\Omega_k}\lambda_m,\lambda_m\rt).
\end{multline*}
\emph{Step 2 (Local term).} For $r>0$ let $(X_i^r,Y_i^r)$ be i.i.d uniformly distributed random variables in $[0,r]^2$ and for $n\le m$ let 
\[
 Fb_r(n,m)=\EE\lt[\Wb_{[0,r]^2}^2\lt(\frac{1}{n}\sum_{i=1}^n \delta_{X_i^r}, \frac{1}{m}\sum_{i=1}^m \delta_{Y_i^r}\rt)\rt]
\]
so that by scaling and Proposition \ref{prop:astbip-dirichlet},
\[
 Fb_r(n,m)=r^2 Fb_1(n,m)\ge r^2 \frac{\log n}{4\pi n}\lt(1+ \frac{n}{m} -\omega(n)\rt).
\]
Define $N_k= n\mu_n(\Omega_k)$ (respectively $M_k=m\lambda_m(\Omega_k)$) to  be the number of points $X_i$ (respectively $Y_i$) in $\Omega_k$ so that in particular $\kappa_k= N_k/M_k$. 
For every given $k$, 
the random variables $N_k$ and $M_k$ are independent binomial random variables with parameters respectively $(n,\rho(\Omega_k))$ and $(m, \rho(\Omega_k))$. We then define the random variables $N^*_k=\min(N_k,M_k)$ and $M^*_k=\max(N_k,M_k)$.

For every fixed $\Omega_k\in V_\delta^r$, recalling that $\Omega_k$ is a cube  of sidelength at most of order $r$, and using \eqref{eq:holder-cube-mn}, we have 
\begin{align*}
 \EE\lt[\Wb^2_{\Omega_k}(\mu_n,\kappa_k \lambda_m)\rt] &\ge (1-C r^\alpha)\EE\lt[\frac{N^*_k}{n}Fb_{\sqrt{|\Omega_k|}}(N^*_k,M^*_k)\rt]\\
 &\ge (1-C r^\alpha)\frac{1}{4\pi n}|\Omega_k|\EE\lt[\log N^*_k \lt(1+ \frac{N^*_k}{M^*_k} -\omega(N^*_k)\rt) \rt]\\
 &\ge (1-C r^\alpha)\frac{ \log (n\rho(\Omega_k))}{4\pi n} |\Omega_k| \lt(1+ \frac{n}{m}- \omega(n r^2 \delta^2)\rt),
\end{align*}
 where in the last line one can argue as for \eqref{eq:FNk}. Summing over $k$  and using \eqref{claim:sumOm} we find 
\begin{equation}\label{local-bi}
 \EE\lt[\sum_{V_\delta^r} \Wb^2_{\Omega_k}(\mu_n,\kappa_k \lambda_m)\rt]\ge (1-C r^\alpha)\frac{ \log n}{ n} \frac{1- C \delta}{4\pi}\lt(1+\frac{n}{m}-\omega(n r^2 \delta^2)\rt)- C\frac{|\log r|}{n}.
\end{equation}

\emph{Step 3 (Global term).} We claim that for $\delta=n^{-\beta}$ with $\beta\in \lt(\frac{1}{3},\frac{1}{2}\rt)$,
\begin{equation}\label{toprovebipglobal}
 \frac{n}{\log n} \EE\lt[W_{\Omega}^2\lt(\sum_{k} \kappa_k \chi_{\Omega_k}\lambda_m,\lambda_m\rt) \chi_A \rt]\les  \frac{\log \log n}{\sqrt{ \log n}}.
\end{equation}
For this we argue along the lines of \cite[Proposition 5.2]{GolTre}. We define
\[
 \theta_k= 1-\kappa_k+\theta
\]
and recall that if \eqref{claim:likely} holds (which is the case if $A$ occurs) then $\frac{3}{2} \theta\ge \theta_k\ge \frac{1}{2}\theta$. We now use triangle inequality to write
\[\begin{split}
  & W_{\Omega}^2\lt(\sum_k \kappa_k \chi_{\Omega_k} \lambda_m, \lambda_m\rt) \\
  & \quad \les W_{\Omega}^2\lt(\sum_k \kappa_k \chi_{\Omega_k} \lambda_m, \sum_k\lt[(1-\theta_k)\lambda_m +\theta \frac{\lambda_m(\Omega_k)}{\rho(\Omega_k)}\rho\rt] \chi_{\Omega_k}\rt)\\
  & \quad + W_{\Omega}^2\lt( \sum_k \lt[(1-\theta_k)\lambda_m +\theta \frac{\lambda_m(\Omega_k)}{\rho(\Omega_k)}\rho\rt] \chi_{\Omega_k}, \sum_k \lt[(1-\theta_k)\lambda_m +\theta_k \frac{\lambda_m(\Omega_k)}{\rho(\Omega_k)}\rho\rt] \chi_{\Omega_k} \rt)\\
  &\quad + W_{\Omega}^2\lt( \sum_k \lt[(1- \theta_k)\lambda_m +\theta_k \frac{\lambda_m(\Omega_k)}{\rho(\Omega_k)}\rho\rt] \chi_{\Omega_k},\lambda_m \rt).
\end{split}\]
The first and last terms are estimated in a similar way so we only estimate the first one. We use the fact that $1-\theta_k = \kappa_k-\theta$ and subadditivity \eqref{eq:sub} to bound from above
\[\begin{split}
  W_{\Omega}^2& \lt(\sum_k \kappa_k \chi_{\Omega_k} \lambda_m, \sum_k  \lt[(\kappa_k-\theta)\lambda_m +\theta \frac{\lambda_m(\Omega_k)}{\rho(\Omega_k)}\rho\rt] \chi_{\Omega_k}\rt)
 \\
 &\quad  \le \sum_k W_{\Omega_k}^2\lt( \bra{ (\kappa_k-\theta) +\theta
  } \lambda_m,(\kappa_k-\theta)\lambda_m + \theta \frac{ \lambda_m(\Omega_k)}{\rho(\Omega_k)}\rho\rt)\\
 &\quad \stackrel{\eqref{eq:sub}}{\le} \theta \sum_k W_{\Omega_k}^2\lt( 
   \lambda_m,  \frac{ \lambda_m(\Omega_k)}{\rho(\Omega_k)}\rho\rt).
 \end{split}
\]
By \eqref{eq:NBr1} and  \eqref{eq:NBr}, we get  
\[
 \EE\lt[W_{\Omega}^2\lt(\sum_k \kappa_k \chi_{\Omega_k} \lambda_m, \sum_k \lt[(\kappa_k-\theta)\lambda_m +\theta \frac{\lambda_m(\Omega_k)}{\rho(\Omega_k)}\rho\rt] \chi_{\Omega_k}\rt]\rt)\les \theta \lt( \frac{\log m }{m} +\delta^3\rt).
\]
For the middle term, we use again subadditivity together with Lemma~\ref{lem:peyre} and the fact  that $\lambda_m(\Omega_k)/\rho(\Omega_k)\sim 1$ in the event $A$, to infer 
\[\begin{split}
   W_{\Omega}^2& \lt( \sum_k \lt[(1-\theta_k)\lambda_m +\theta \frac{\lambda_m(\Omega_k)}{\rho(\Omega_k)}\rho\rt] \chi_{\Omega_k}, \sum_k \lt[(1-\theta_k)\lambda_m +\theta_k \frac{\lambda_m(\Omega_k)}{\rho(\Omega_k)}\rho\rt] \chi_{\Omega_k} \rt)
  \\
  &\quad  \le  W_{\Omega}^2\lt( \sum_k \theta \frac{\lambda_m(\Omega_k)}{\rho(\Omega_k)} \chi_{\Omega_k} \rho, \sum_k \theta_k \frac{\lambda_m(\Omega_k)}{\rho(\Omega_k)} \chi_{\Omega_k} \rho\rt)\\
  &\quad \les \frac{1}{\theta} \lt\|\sum_k (\theta-\theta_k)\frac{\lambda_m(\Omega_k)}{\rho(\Omega_k)} \chi_{\Omega_k} \rho\rt\|_{H^{-1}}^2\\
  &\quad =  \frac{1}{\theta} \lt\|\sum_k (\kappa_k-1)\frac{\lambda_m(\Omega_k)}{\rho(\Omega_k)} \chi_{\Omega_k} \rho\rt\|_{H^{-1}}^2\\
  &\quad \les \frac{1}{\theta} \lt\|\sum_k \lt(\frac{\mu_n(\Omega_k)}{\rho(\Omega_k)}-1\rt) \chi_{\Omega_k} \rho\rt\|_{H^{-1}}^2+\frac{1}{\theta} \lt\|\sum_k \lt(\frac{\lambda_m(\Omega_k)}{\rho(\Omega_k)}-1\rt) \chi_{\Omega_k} \rho\rt\|_{H^{-1}}^2,
\end{split}\]
where in the last line we used that $(\kappa_k-1)\frac{\lambda_m(\Omega_k)}{\rho(\Omega_k)}=(\frac{\mu_n(\Omega_k)}{\rho(\Omega_k)}-1)-(\frac{\lambda_m(\Omega_k)}{\rho(\Omega_k)}-1)$ and triangle inequality in $H^{-1}$.
By \eqref{eq:global} we obtain 
\[
  \EE\lt[\lt\|\sum_i (1-\kappa_k)\frac{\lambda_m(\Omega_k)}{\rho(\Omega_k)} \chi_{\Omega_k} \rho\rt\|_{H^{-1}}^2\rt] \les \frac{|\log r|}{n}.\]
This proves
\[
 \EE\lt[W_{\Omega}^2\lt(\sum_i \kappa_k \chi_{\Omega_k} \lambda_m, \lambda_m\rt) \chi_{A}\rt]\les\theta \lt( \frac{\log n }{n} +\delta^3\rt) +\frac{1}{\theta}\frac{|\log r|}{n},
\]
which recalling the choice \eqref{choicetheta} of $\theta$, $\delta$ and $r$ concludes the proof of \eqref{toprovebipglobal}.

\emph{Step 4 (Conclusion).} Putting together \eqref{local-bi} and \eqref{toprovebipglobal} and using \eqref{claim:smallprob},  we find that 
\[ \begin{split}
\frac{n}{\log n} \EE\lt[\Wb^2_{\Omega}(\mu_n,\lambda_m)\rt]& \ge (1-\eps) (1-C r^\alpha)\frac{1- C \delta}{4\pi }\lt(1+\frac{n}{m}-\omega(n r^2 \delta^2)\rt) \\
& \qquad - \frac{C}{\eps} \bra{ \frac{\log \log n}{\sqrt{\log n}} +\exp\lt(-c \frac{n^{1-2\beta}}{(\log n)^3}\rt)} \\
& \ge \frac{1}{4\pi} \bra{1+ \frac{m}{n} + \omega(n)},
\end{split}\]
where in the last line we  chose $\eps = \eps(n)\to 0$  as $ \to \infty$ with $\eps \gg \log \log n /\sqrt{\log n}$. 

\section{Extension to general settings}\label{sec:manifolds}

In this section we show that the techniques developed above lead to extensions of Theorem~\ref{thm:main-ref} and Theorem~\ref{thm:main-ref-bip} to more general settings, including Riemannian manifolds.
Let us give the following definition (see also \cite[Definition 3.1]{trillos2015rate}).

\begin{definition}
 We say that a metric measure space $(\Omega, d, \meas)$ with $\meas(\Omega) = 1$ is well-decomposable if for every $\eps>0$ there exist a \emph{finite} family, that we call $\eps$-decomposition, $(U_k, \Omega_k, T_k)_{k=1}$, such that, for every $k$,
\begin{enumerate}
\item $U_k \subseteq \Omega$ is open and $\meas(U_k \cap U_k) = 0$ for $k \neq k$, $\meas( \bigcup_k U_k) = 1$,
\item $\Omega_k\subseteq\R^2$ is open, bounded and connected, with Lipschitz boundary,
\item $T_k: \overline{U}_k \to \overline{\Omega}_k$, is invertible with $T_k(U_k) = \Omega_k$ and $T_k(\partial U_k) = \partial \Omega_k$ and 
\[ \Lip T_k, \Lip T_k^{-1} \le (1+\epsilon), \]
\item $(T_k)_\sharp (\meas\restr U_k)$ has H\"older continuous density, uniformly positive and bounded.
\end{enumerate}
\end{definition}

If $(\Omega, d,\meas)$ is well-decomposable we define 
\[ |\Omega| =  \inf_{\eps >0} \cur{ \sum_{k} |\Omega_k|\, : \, \text{$(U_k, \Omega_k, T_k)_k$ is an $\eps$-decomposition of $\Omega$}}.\]
\begin{remark}\label{rem:vol}
 Notice that the quantity inside the brackets is a decreasing function of $\eps$ and that if $\Omega$ is a Riemannian manifold it coincides with its (Riemannian) volume.
\end{remark}

In addition to such decomposability assumption, we need to assume the validity of the inequality:
\begin{equation}\label{eq:wasserstein-poincare} W_{2}^2( f\meas, \meas) \les \nor{ f - 1}_{L^2(\meas)}^2, \quad \text{for every probability density $f$.}  \end{equation}
This estimate is analog to the weaker form of the combination of \eqref{eq:estimCZ} with Lemma \ref{lem:Sob} used in \cite{BeCa,GolTre} and is intimately connected to the validity of an
 $L^2$-Poincar\'e-Wirtinger inequality (see e.g. \cite[Lemma 3.4]{GolTre}).

\begin{theorem}\label{thm:main-abstract}
Let $(M,d)$ be a length space and  $\Omega \subseteq M$ be open with $\diam(\Omega)<\infty$. Let $\meas$ be a probability measure on $M$ with $\meas(\Omega) = 1$, such that \eqref{eq:wasserstein-poincare} holds and 
 $(\Omega, d,\meas)$ is well-decomposable and consider $(X_i, Y_i)_{i=1}^\infty$ be i.i.d.\ with common density $\meas$. Then, for every $q \in [1,\infty]$, 
\begin{equation}\label{eq:matching-mms-bip} \lim_{\substack{n,m \to \infty \\ m/n \to q}} \frac{n}{\log n} \mathbb{E} \sqa{ W_2^2\bra{ \frac 1 n \sum_{i=1}^n \delta_{X_i}, \frac 1 m \sum_{j=1}^m \delta_{Y_j} }}  = \frac {|\Omega|} {4 \pi}\bra{ 1 + \frac{1}{q}},
\end{equation}
and
\begin{equation}\label{eq:matching-mms-ref} \lim_{n \to \infty} \frac{n}{\log n} \mathbb{E} \sqa{ W_2^2\bra{ \frac 1 n \sum_{i=1}^n \delta_{X_i}, \meas }}  = \frac {|\Omega|} {4 \pi}.\end{equation}
\end{theorem}

Examples of  metric measure spaces for which the result apply include smooth domains in Riemannian manifolds with boundary, as shown in Appendix \ref{app:wd}. But also other, less regular cases, may be included, e.g.\  polyhedral surfaces.

\begin{proof}[Proof of Theorem~\ref{thm:main-abstract}]
As usual, we simply write $W = W_2$ and $\Wb = \Wb_2$. We write $\mu_n =  \frac 1 n \sum_{i=1}^n \delta_{X_i}$, $\lambda_m = \frac 1 m \sum_{j=1}^m \delta_{Y_j}$. For $\eps>0$, let  $(U_k, \Omega_k, T_k)_{k}$ be an $\eps$-decomposition and let $\kappa_k = \mu_n(U_k)/\lambda_m(U_k)$. In both cases, after an application of the sub-additivity and super-additivity inequalities \eqref{eq:mainsub}, \eqref{eq:mainsup} we are reduced to estimate separately a finite sum of ``local'' terms, one for each $U_k$, and the global terms (using that $\Wb \le W$) that are respectively
\begin{equation}\label{eq:mms-global}
\EE\sqa{ W_2^2\bra{ \sum_{k} \kappa_k \chi_{U_k} \lambda_m,   \lambda_m } } \quad \text{and} \quad \EE\sqa{ W_2^2\bra{ \sum_{k}\frac{\mu_n(U_k)}{\meas(U_{k})} \chi_{U_k} \meas,   \meas } }.
\end{equation}

\emph{Local terms.}  Using \eqref{eq:W-T} with $T_k^{-1}$ and \eqref{eq:W-T} with $T_k$ and taking expectation give respectively
\[  \EE\sqa{ W^2\bra{ \mu_n \restr U_k,  \kappa_k  \lambda_m \restr U_k} }\le (1+\eps)^{2} \EE\sqa{ N W^2\bra{ \frac{1}{N} \sum_{i=1}^{N} \delta_{X_{k,i}}, \frac{1}{M} \sum_{j=1}^{M} \delta_{Y_{k,j}}}},\]
and
\[
 \EE\sqa{ \Wb_{U_i}^2\bra{ \mu_n \restr U_k, \kappa_k \lambda_m \restr U_k} }\ge (1+\eps)^{-2} \EE\sqa{ N \Wb_{\Omega_i}^2\bra{ \frac{1}{N} \sum_{i=1}^{N} \delta_{X_{k,i}}, \frac{1}{M} \sum_{j=1}^{M} \delta_{Y_{k,j}}}},
\] 
where, for every $k$, we consider i.i.d.\ random variables $(X_{k,i}, Y_{k,i})_{i=1}^\infty$ with common law $\rho_k$ given by $(T_k)_\sharp \meas$ normalized to a probability measure on $\Omega_k$ and $N$, $M$ are further independent  random variables with binomial laws of parameters respectively $(n, \meas(U_k))$, $(m, \meas(U_k))$. For each $k$, by conditioning upon $N$, $M$,  and by 
\eqref{eq:lower-bound-proof-bip}, \eqref{eq:upper-bound-proof-bip} with $(X_{k,i}, Y_{k,i})_{i=1}^\infty$, we obtain that, letting $N^* = \max\cur{N,M}$, $N_* = \min \cur{N,M}$
\[ \EE\sqa{ N W^2\bra{ \frac{1}{N} \sum_{i=1}^{N} \delta_{X_{k,i}}, \frac{1}{M} \sum_{j=1}^{M} \delta_{Y_{k,j}}}} \le \EE\sqa{ N \frac{|\Omega_k|\log N_*}{4\pi N_*} \bra{1+ \frac{N_*}{N^*} + \omega_k(N_*)}},\]
and similarly
\[ \EE\sqa{ N \Wb_{\Omega_k}^2\bra{ \frac{1}{N} \sum_{i=1}^{N} \delta_{X_{k,i}}, \frac{1}{M} \sum_{j=1}^{M} \delta_{Y_{k,j}}}} \ge \EE\sqa{ N \frac{|\Omega_k|\log N_*}{4 \pi N_*} \bra{1+ \frac{N_*}{N^*} - \omega_k(N_*)}}.\]
As $n \to \infty$ with $m/n \to q \in [1, \infty]$, we have
\[ N/n \to \meas(U_k), \quad  M/m \to \meas(U_k) \quad \text{and} \quad N_*/N \to 1, \quad N^*/N \to q,\]
with the usual concentration inequalities which eventually give
\[  \lim_{n \to \infty} \frac {n}{\log n} \EE\sqa{ N \frac{|\Omega_k|\log N_*}{4 \pi N_*} \bra{1+ \frac{N_*}{N^*} - \omega_k(N_*)}} = \frac{|\Omega_k|}{4 \pi}\bra{1 + \frac{1}{q}}.\]
A similar argument gives
\begin{equation}\label{eq:local-mms} \lim_{n \to \infty} \frac {n}{\log n} \EE\sqa{ W_{U_k}^2\bra{ \frac{\mu_n(U_k)}{\meas(U_{k})} \meas,   \meas } } = \frac{|\Omega_k|}{4 \pi}.\end{equation}

\emph{Global term.}  We consider the second term in \eqref{eq:mms-global}, for which we apply \eqref{eq:wasserstein-poincare} obtaining
\[ \EE\sqa{ W_2^2\bra{ \sum_{k}\frac{\mu_n(U_k)}{\meas(U_{k})} \chi_{U_k} \meas,   \meas } } \les \sum_{k} \EE\sqa{ \bra{ \frac{\mu_n(U_k)}{\meas(U_{k})} - 1 }^2} \meas(U_k) \les \frac{C(\eps)}{n}.\]
After multiplying by $n/ \log n$ and letting first $n \to \infty$ and then $\eps \to 0$, the validity of \eqref{eq:matching-mms-bip} is thus settled.

To bound \eqref{eq:mms-global}, we let $\kappa = \min_{k}\kappa_k -\eps$, so that, as $n \to \infty$, $\kappa \to 1-\eps$  and using union and Chernoff bounds  we have the crude estimate (but sufficient for our purposes)
\begin{equation}\label{eq:chernoff-mms} P(\forall k, \, |\kappa- \kappa_k| > 2\epsilon ) \lesssim  \frac {C(\eps)} {n}.
\end{equation}
We bound from above, using the subadditivity inequality \eqref{eq:sub} and the triangle inequality
\begin{equation}\label{eq:three-terms-mms} \begin{split}  \EE\sqa{ W_2^2\bra{ \sum_{k} \kappa_k \chi_{U_k} \lambda_m,   \lambda_m } } & \le \EE\sqa{ W_2^2\bra{ \sum_{k} (\kappa_k-\kappa) \chi_{U_k} \lambda_m,  (1-\kappa) \lambda_m } }\\
&  \le  \EE\sqa{ W_2^2\bra{ \sum_{k}(\kappa_k-\kappa) \chi_{U_k} \lambda_m, \sum_{k}(\kappa_k-\kappa) \frac{\lambda_m(U_k)}{\meas(U_k)} \chi_{U_k} \meas}} \\
& \quad  + \EE\sqa{ W_2^2\bra{ \sum_{k}(\kappa_k-\kappa)  \frac{\lambda_m(U_k)}{\meas(U_k)} \chi_{U_k} \meas, (1-\kappa)\meas } }\\
& \quad  +  \EE\sqa{ W_2^2\bra{ (1-\kappa) \meas, (1-\kappa) \lambda_m } }.
\end{split}\end{equation}
For the first term, we use again subadditivity \eqref{eq:sub},
\[\begin{split} \EE& \sqa{ W_2^2\bra{ \sum_{k}(\kappa_k-\kappa) \chi_{U_k} \lambda_m, \sum_{k}(\kappa_k-\kappa)  \frac{\lambda_m(U_k)}{\meas(U_k)} \chi_{U_k} \meas}} \\
& \quad \le \sum_{k} \EE\sqa{ W_{U_k}^2\bra{ (\kappa_k-\kappa) \lambda_m, (\kappa_k-\kappa)  \frac{\lambda_m(U_k)}{\meas(U_k)} \meas}}\\
& \quad \le \sum_{k}  C(\eps) \frac{\diam(E)^2}{n^4} + \eps \EE\sqa{ W_{U_k}^2\bra{ \lambda_m,   \frac{\lambda_m(U_k)}{\meas(U_k)} \meas} } \\
&\quad \les C(\eps) \frac{\diam(E)^2}{n} + \eps \sum_{k} |\Omega_k| \frac{\log m}{m},\end{split}\]
having used \eqref{eq:chernoff-mms} and \eqref{eq:local-mms} with $\lambda_m$ instead of $\mu_n$. For the second term in \eqref{eq:three-terms-mms}, we use \eqref{eq:wasserstein-poincare}, 
\[\begin{split}   \EE& \sqa{ W_2^2\bra{ \sum_{k}(\kappa_k-\kappa)  \frac{\lambda_m(U_k)}{\meas(U_k)} \chi_{U_k} \meas, (1-\kappa)\meas } } \\
&\quad \les \sum_{k}   \EE\sqa{ \bra{ (\kappa_k-\kappa)  \frac{\lambda_m(U_k)}{\meas(U_k)}  -  (1-\kappa)}^2} \meas(U_k)  \\
& \quad  \les \frac{ C(\eps)}{n},\end{split} \]
having used the variance bounds 
\[ \EE\sqa{\bra{ \frac{\lambda_m(U_k)}{\meas(U_k)} - 1 }^2} + \EE\sqa{ (\kappa_k - 1)^2} \les \frac{ C(\eps)}{n}.\]
For the third term in \eqref{eq:three-terms-mms}, we use that $1-\kappa > 2\eps$ with probability smaller than $C(\eps)/n$ and then on the complementary event, the already settled \eqref{eq:matching-mms-ref}, to obtain
\[ \EE\sqa{ W_2^2\bra{ (1-\kappa) \meas, (1-\kappa) \lambda_m } } \les \frac{C(\eps)}{n}+  \eps \frac{m}{\log m}.\]
After collecting all these bounds, we multiply by $n/ \log n$ and let first $n \to \infty$ and then $\eps \to 0$ to conclude. 
\end{proof}

\section{Proof of Corollary~\ref{cor:assignment}}\label{sec:ass}
We recall that with our usual notation, we want to prove that  for any sequence $m= m(n) \ge n$ with $\lim_{n \to \infty} (m - n)/ \log n = 0$, it holds
\[ \lim_{n \to \infty } \frac{1}{\log n } \mathbb{E} \sqa{\min_{\sigma \in \mathcal{S}_{n,m}} \sum_{i=1}^n |X_i - Y_{\sigma(i)}|^2
}  = \frac {|\Omega|} {2 \pi}.\]
The inclusion $\mathcal{S}_n \subseteq \mathcal{S}_{n,m}$ yields immediately
\[  \min_{\sigma \in \mathcal{S}_{n,m}} \sum_{i=1}^n |X_i - Y_{\sigma(i)}|^2 \le  \min_{\sigma \in \mathcal{S}_{n}} \sum_{i=1}^n |X_i - Y_{\sigma(i)}|^2,\]
so that it always holds, for any sequence $m=m(n) \ge n$,
\[ \limsup_{n \to \infty} \frac{1}{\log n} \EE\sqa{  \min_{\sigma \in \mathcal{S}_{n,m}} \sum_{i=1}^n |X_i - Y_{\sigma(i)}|^2} \le  \lim_{n \to \infty} \frac{1}{\log n} \EE\sqa{  \min_{\sigma \in \mathcal{S}_{n}} \sum_{i=1}^n |X_i - Y_{\sigma(i)}|^2} = \frac{|\Omega|}{2 \pi}.\]

To show the converse inequality, given $\sigma \in \mathcal{S}_{n,m}$ we induce a matching between $\bra{X_i}_{i=1}^m$ and $\bra{Y_j}_{j=1}^m$ by pairing the points $\bra{X_i}_{i=n+1}^m$ to those in $\bra{Y_j}_{j=1}^m \setminus \bra{Y_{\sigma(i)}}_{i=1}^n$. Since $\Omega$ is bounded, we obtain the inequality
\[ \min_{\sigma \in \mathcal{S}_{m}} \sum_{i=1}^m |X_i - Y_{\sigma(i)}|^2 \le  \min_{\sigma \in \mathcal{S}_{n,m}} \sum_{i=1}^n |X_i - Y_{\sigma(i)}|^2 + \diam(\Omega)^2 (m-n).\]
Taking expectation and using the assumption $(m-n)/\log n \to 0$ which in particular gives $\log n / \log m \to 1$ yields
\[ \frac{|\Omega|}{2 \pi} = \lim_{m \to \infty} \frac{1}{\log m} \EE\sqa{  \min_{\sigma \in \mathcal{S}_{m}} \sum_{i=1}^m |X_i - Y_{\sigma(i)}|^2} \le \liminf_{n\to \infty} \frac{1}{\log n} \EE\sqa{  \min_{\sigma \in \mathcal{S}_{n,m}} \sum_{i=1}^n |X_i - Y_{\sigma(i)}|^2}.\]

\section{Proof of Corollary~\ref{cor:main-TSP}}\label{sec:TSP}

Given $\bra{x_i}_{i=1}^n, \bra{y_i}_{i=1}^n  \subseteq \R^2$, we introduce the notation
\[ \C_{TSP}^2(\bra{x_i}_{i=1}^n ) = \min_{\tau \in \mathcal{S}_n} \sum_{i=1}^{n} |x_{\tau(i)} - x_{\tau(i+1)}|^2\]
for the costs of the Euclidean travelling salesperson problem and
\[  \C_{\bTSP}^2( \bra{x_i}_{i=1}^n, \bra{y_i}_{i=1}^n ) = \min_{\tau, \tau' \in \mathcal{S}_n} \sum_{i=1}^{n} |x_{\tau(i)} - y_{\tau'(i)}|^2 + |y_{\tau'(i)} - x_{\tau(i+1)}|^2\]
for the cost of its bipartite variant (writing $\tau(n+1) = \tau(1)$ for permutations $\tau \in \mathcal{S}_n$). 




The proof of the lower bound follows straightforwardly taking expectation in the inequality
\begin{equation}\label{eq:lower-bound-tsp} \C_{\bTSP}^2(\bra{X_i}_{i=1}^n, \bra{Y_i}_{i=1}^n) \ge 2 \min_{\sigma \in \mathcal{S}_n} \sum_{i=1}^n |X_i - Y_{\sigma(i)}|^2.\end{equation}
Using Theorem~\ref{thm:main-ref-bip}, we have then
\[ \liminf_{n \to \infty} \frac{1}{\log n} \mathbb{E} \sqa{  \C_{\bTSP}^2 \bra{ \bra{X_i}_{i=1}^n, \bra{Y_i}_{i=1}^n }} \ge \frac {|\Omega|} {\pi}.\]

To prove the converse inequality, let $\tau\in \mathcal{S}_n$ be a minimizer for $\C_{\TSP}^2(\bra{X_i}_{i=1}^n)$ and let $\sigma \in \mathcal{S}_n$ be an optimal matching between $\bra{X_i}_{i=1}^n$ and $\bra{Y_i}_{i=1}^n$, with respect to the quadratic cost. We let $\tau' = \sigma \circ \tau \in \mathcal{S}_n$, so that
\[
 \C_{\bTSP}^2(\bra{X_i}_{i=1}^n, \bra{Y_i}_{i=1}^n)  \le \sum_{i=1}^n  |X_{\tau(i)} - Y_{\sigma(\tau(i))}|^2 + |Y_{\tau(i)} - X_{\tau(i+1)}|^2.\]
 For every $\eps>0$, using the inequality $|a+b|^2 \le (1+\eps) |a|^2 + (1+\eps^{-1})|b|^2$, we bound from above, for every $i\in \cur{1, \ldots, n}$,
 \[ \begin{split} |Y_{\tau(i)} - X_{\tau(i+1)}|^2 & =  |(Y_{\sigma( \tau(i))} - X_{\tau(i)}) + ( X_{\tau(i)} - X_{\tau(i+1)}) |^2 \\
 & \le (1+\eps) |Y_{\sigma( \tau(i))} - X_{\tau(i)}|^2 + (1 +\eps^{-1}) | X_{\tau(i)} - X_{\tau(i+1)} |^2.
 \end{split}\]
 Summing upon $i$ gives
 \[ \begin{split} \C_{\bTSP}^2 (\bra{X_i}_{i=1}^n, \bra{Y_i}_{i=1}^n)  \le   & (2+\eps) \min_{\sigma \in \mathcal{S}_n} \sum_{i=1}^n |X_i - Y_{\sigma(i)}|^2 \\
 & \quad  + (1+\eps^{-1})\C_{\TSP, 2}^2(\bra{X_i}_{i=1}^n).\end{split} \]
We claim that
 \begin{equation}\label{eq:tsp-mono-upper}
\limsup_{n \to \infty}\frac 1 {\log n} \EE\sqa{ \C_{\TSP}^2(\bra{X_i}_{i=1}^n) }  = 0,
 \end{equation}
 so that we obtain, again by Theorem~\ref{thm:main-ref-bip},
\[   \limsup_{n \to \infty} \frac{1}{\log n} \EE \sqa{  \C_{\bTSP}^2 \bra{ \bra{X_i}_{i=1}^n, \bra{Y_i}_{i=1}^n }} \le (2+\eps) \frac {|\Omega|} {2\pi}.\]
The thesis follows letting $\eps \to 0$.

To prove \eqref{eq:tsp-mono-upper}, we rely on the general bound
\[ \sup_{ \bra{x_i}_{i=1}^n \subseteq \Omega } \C_{\TSP}^2(\bra{x_i}_{i=1}^n) \les 1,\]
where the implicit constant depends on $\Omega$ only. Its proof rely on the  space-filling curve heuristic \cite[Section 2.6]{steele1997probability}, that we briefly recall here for the reader convenience. Without loss of generality, we consider the case $\Omega = Q$ a square. Consider a Peano curve $\psi: [0,1] \to \overline{Q}$, that is surjective and $1/2$-H\"older continuous. Let $t_i \in [0,1]$ be such that $\psi(t_i) = x_i$ and choose 
$\tau\in \mathcal{S}_n$ such that $(t_{\tau(i)})_{i=1}^n$ are in increasing order. Then,
\[
\C_{\TSP}^2(\bra{x_i}_{i=1}^n)  \le \sum_{i=1}^n |\psi(t_{\tau(i)}) - \psi(t_{\tau(i+1)})|^2 
 \le \nor{\psi}_{C^{1/2}}  \sum_{i=1}^n |t_{\tau(i)} - t_{\tau(i+1)}|  \le 2 \nor{\psi}_{C^{1/2}}.
\]

\begin{remark}
The argument above is in fact rather general and it applies on general metric spaces provided that the cost of the quadratic bipartite matching problem is asymptotically larger than that of the travelling salesperson problem. Also, already in the Euclidean setting, it seems possible to adapt it to the case $p \neq 2$, but existence of the limit for the asymptotic cost of the bipartite matching problem is not known. Finally, exactly as in \cite{capelli2018exact}, the argument applies as well for the random bipartite quadratic $2$-factor problem, i.e., for the relaxation where the single tour is replaced by a collection of  disjoint ones: indeed, since the cost becomes smaller, it is sufficient to notice that the lower bound \eqref{eq:lower-bound-tsp} still holds.  
\end{remark}
\appendix

\section{Proof of Lemma~\ref{lem:decomp}}\label{app:proof-lemma-dec}
Let us recall that given a bounded open set  $\Omega$ with Lipschitz boundary, and $\{Q_k\}_k$  a Whitney decomposition of $\Omega$ we want to construct for every $\delta>0$
small enough (depending on $\Omega$), a partition of $U_\delta\cup V_\delta$ of $\Omega$ such that   $V_\delta=\{Q_k \ : \ \diam(Q_k) \ge \delta\}$, and $U_\delta=\{\Omega_k\}_k$ with 
$\Omega_k$ open, $\diam(\Omega_k)\les \delta$ and $|\Omega_k|\sim \delta^d$.

 We start by constructing a partition of the set $A_\delta=\{d(\cdot, \Omega^c)< \sqrt{d} \delta\}$ by open sets $\widetilde{\Omega}_k$ satisfying $\diam(\widetilde{\Omega}_k)\les \delta$, $|\widetilde{\Omega}_k|\sim \delta^d$.
 For this we first consider a $\delta-$net $\{x_k\}_k$ of $\partial \Omega$, that is a family satisfying $\partial \Omega\subset \cup_i B_\delta(x_i)$ and $\min_{i\neq j} |x_i-x_j|\ge \frac{\delta}{2}$. Such a family can be for instance constructed by choosing any starting point $x_1\in \partial \Omega$ and then setting $x_k\in \textrm{argmax}_{\partial \Omega} d(x, \cup_{i=1}^{d-1}\{x_i\})$ as long as $\partial \Omega$ is not covered by $\cup_{i=1}^{k-1} B_\delta(x_i)$. We then set 
 \[
  \widetilde{\Omega}_k=\{ x\in \Omega \ : |x-x_k|< |x-x_i|\quad \forall i\neq k\}\cap A_\delta.
 \]
Notice that since $\{|x-x_k|= |x-x_i|\}$ is a hyperplane and thus of Lebesgue measure zero, $\{\widetilde{\Omega}_k\}_k$ is indeed a partition of $A_\delta$ in disjoint open sets. 
Notice also that by definition of a Whitney partition, if $Q_k\in V_\delta$ then $Q_k\cap A_\delta=\emptyset$. Now on the one hand, by triangle inequality it
is immediate that $\diam(\widetilde{\Omega}_k)\les \delta$ and thus also $|\widetilde{\Omega}_k|\les \delta^d$. On the other hand, still by triangle inequality, 
$B_{\frac{\delta}{4}}(x_k)\cap \Omega\subset \widetilde{\Omega}_k$ so that by Lipschitz regularity of $\Omega$, we also obtain $|\widetilde{\Omega}_k|\ges \delta^d$.\\
We finally define $U_\delta=\cup_k\{ Q_k \notin V_\delta \ : \ d(Q_k,\Omega^c)>\sqrt{d} \delta\} \cup_k \{\Omega_k\}$ where 
\[
 \Omega_k=\widetilde{\Omega}_k\cup \bigcup_j\lt\{Q_j \ : d(Q_j,\Omega^c)<\sqrt{d} \delta, \ Q_j\cap A_\delta^c\neq \emptyset \, \textrm{ and } \, k=\textrm{argmin}\{i \ : \ \widetilde{\Omega}_i \cap Q_j\neq \emptyset\} \rt\}. 
\]
In words this means that we consider in $U_\delta$ either cubes which are at distance greater than $\sqrt{d} \delta$ from $\Omega^c$ but which are not in $V_\delta$
or we combine the sets $\widetilde{\Omega}_k$ with the cubes which intersect both $A_\delta$ and its complement. The choice to which $\widetilde \Omega_k$
we associate $Q_i$ is arbitrary (as long as they intersect). The sets $\Omega_k$ satisfy both $\diam(\Omega_k)\les \delta$ and $|\Omega_k|\sim \delta^d$ 
since for every $j$ such that $d(Q_j,\Omega^c)<\sqrt{d} \delta$ and  $ Q_j\cap A_\delta^c\neq \emptyset$, we have  $\ell(Q_j)\sim \delta$ and thus the number of such cubes which can intersect $\Omega_k$ is uniformly bounded with respect to $\delta$.

\section{Decomposition of Riemannian manifolds}\label{app:wd}

We show that Theorem~\ref{thm:main-abstract} applies to compact connected smooth Riemannian manifolds $(M,g)$, possibly with boundary. In fact, we argue in the case of bounded connected domains $\Omega\subseteq M$ with smooth boundary (so that $M$ itself does not need to be compact). Notice that $\Omega$ with the restriction of the metric $g$ is also a compact connected smooth Riemannian manifolds with boundary, but the induced distance would then be larger (and different, in general) than the restriction of the one on $M$, exactly as in the Euclidean case.

\begin{lemma}\label{lem:wd}
Let $(M,g)$ be a Riemannian manifold with (possibly empty) boundary  $\partial M$ and let $\Omega \subseteq M$ be open, bounded, connected, with smooth boundary. Let also $\rho$ be a H\"older continuous probability density on $\Omega$ (with respect to the Riemannian volume) uniformly strictly positive and bounded from above. Then, $(\Omega, d, \rho)$, where $d$ is the restriction of the Riemannian distance on $M$, is well-decomposable and \eqref{eq:wasserstein-poincare} holds.
\begin{proof}
To show that $\Omega$ is well-decomposable, we modify the well-known argument to obtain a decomposition of $M$ into geodesic triangles, see e.g.\ \cite[Theorem 2.3.A.1]{jost2013compact}, to take into account also the presence of the boundary $\partial \Omega$. To simplify the exposition, assume first that $\partial M = \emptyset$. Then, by compactness, $\partial \Omega$ is the union of a finite family of closed simple $C^1$ curves, and to simplify again let us assume that there is only one such curve $\gamma: \mathbb{S}^1 \to M$ parametrized so that $\dot{\gamma}(t) \neq 0$ for every $t\in \mathbb{S}^1$. 
Consider also a parametrization of the tubular neighbourhood of $\gamma(\mathbb{S}^1)$, i.e., a map $\Gamma : \mathbb{S}^1 \times (-r_0, r_0) \to M$,
\[ \Gamma(t, r)  = \exp_{\gamma(t)} ( r \dot{\gamma}(t)^{\perp}),\]
where $\dot{\gamma}(t)^{\perp}$ denotes the inward unit normal to $\partial  \Omega$  at $\gamma(t)$. By compactness, if $r_0$ is sufficiently small, such a parametrization exist, is smooth (in  particular Lipschitz) and
\[ \Gamma(\mathbb{S}^1 \times (0, r_0)) \subseteq \Omega \quad \text{while} \quad \Gamma(\mathbb{S}^1 \times (- r_0,0)) \subseteq  \Omega ^c.\]
Let us also assume that $r_0$ is sufficiently small so that $\exp_p$ is well-defined and invertible on a ball $B_p(r_0)$ for every $p \in \overline{\Omega}$ and with Lipschitz constant smaller than $1+\eps$ (together with its inverse). Then, normal coordinates $T_p = \exp_p^{-1}$ are well defined on the image  $\exp_p\bra{ B_p(r_0)}$.

Consider a finite mesh $\cur{t_i}_{i} \subseteq \mathbb{S}^1$ with $d(t_i, t_{i+1}) < r_0 /\Lip(\Gamma)$ and let $p_i = \gamma(t_i)$ and $q_i = \Gamma(t_i, r_0/2) \in \Omega$. Then, $d(q_i, q_{i+1}) \le \Lip(\Gamma) |t_i-t_{i+1}| < r_0/4$. On the other side, $d(q_i, \partial \Omega) \ge r_0/2$, so that $B(q_i, r_0/4) \cap \partial \Omega = \emptyset$ and in particular the geodesic segment connecting $q_i$ to $q_{i+1}$ does not intersect $\partial \Omega$. We define $U_{i}$ to be the rectangle-like region with boundaries given by the geodesic segments connecting $p_i$ to $q_i$, then $q_i$ to $q_{i+1}$ and $q_{i+1}$ and $p_i$, together with the piece of boundary between $p_i$ and $p_{i+1}$ (i.e., $\Gamma(0, s)$ with $s \in [t_i, t_{i+1}]$. We also notice that the boundary is piecewise $C^1$ (with non-tangential intersections of different pieces), hence when transformed by $T_{p_i}$ it is Lipschitz regular. 

This settles the tubular neighbourhood of $\partial  \Omega$. To complete the construction we then consider a finite set of points $\cur{p_j} \subseteq \Omega\setminus \Gamma(\mathbb{S}^1, (0, r_0/2))$ such that for every $q \in \Omega$ there exist $p_j$ with $d(p_j, q)< r_0$ and for every $p_j$ there are distinct points $p_k$, $p_\ell$ such that
\[ \max \cur{ d(p_j, p_k), d(p_k, p_\ell),  d(p_j, p_\ell)} <r_0/2. \]
We connect with geodesic segments all points $p_j$, $p_k$ with $d(p_j, p_k)< r_0/2$ as well as $p_j$, $q^+_i$ with $d(p_j, q^+_i)<r_0/2$. Notice that none of these segments intersects $\partial \Omega$. As in the proof of \cite[Theorem 2.3.A.1]{jost2013compact},
up to enlarging the family to include intersections of these segments, we obtain a decomposition into geodesic polygons. To conclude the construction of the $\eps$-decomposition it is
sufficient to add these polygons to the sets $U_i$ obtained above. 

In the case of non-empty $\partial M$, the only difference is that the notion of tubular neighbourhood  must take into account points $p = \gamma(t) \in \Omega \cap \partial M$, hence $\Gamma(t,r)$ is defined only for $r\in [0,r_0)$.

Finally,  inequality \eqref{eq:wasserstein-poincare} follows from the validity of an $L^2$-Poincar\'e-Wirtinger inequality on $\Omega$, adapting the proof of \cite[Lemma 3.4]{GolTre} to  the Riemannian setting. 
In turn, the validity of Poincar\'e-Wirtinger inequality seems to be folklore on smooth  connected compact Riemannian manifolds and a full proof, also in presence of a smooth boundary, may be given by gluing local inequalities \cite[Section 10.1]{hajlasz2000sobolev}. 
\end{proof}
\end{lemma}




\bibliographystyle{acm}

\bibliography{OT}

\begin{thebibliography}{10}

\bibitem{AKT84}
{\sc {Ajtai}, M., {Koml\'os}, J., and {Tusn\'ady}, G.}
\newblock {On optimal matchings.}
\newblock {\em {Combinatorica} 4\/} (1984), 259--264.

\bibitem{AmGlau}
{\sc Ambrosio, L., and Glaudo, F.}
\newblock Finer estimates on the 2-dimensional matching problem.
\newblock {\em J. \'{E}c. polytech. Math. 6\/} (2019), 737--765.

\bibitem{ambrosio2019optimal}
{\sc Ambrosio, L., Glaudo, F., and Trevisan, D.}
\newblock On the optimal map in the $2$-dimensional random matching problem.
\newblock {\em Discrete \& Continuous Dynamical Systems-A 39}, 12 (2019), 7291.

\bibitem{AST}
{\sc Ambrosio, L., Stra, F., and Trevisan, D.}
\newblock A {PDE} approach to a 2-dimensional matching problem.
\newblock {\em {Probab. Theory Relat. Fields} 173}, 1-2 (2019), 433--477.

\bibitem{arjovsky2017wasserstein}
{\sc Arjovsky, M., Chintala, S., and Bottou, L.}
\newblock Wasserstein generative adversarial networks.
\newblock In {\em International conference on machine learning\/} (2017), PMLR,
  pp.~214--223.

\bibitem{ArmSer}
{\sc Armstrong, S., and Serfaty, S.}
\newblock Local laws and rigidity for {C}oulomb gases at any temperature.
\newblock {\em Ann. Probab. 49}, 1 (2021), 46--121.

\bibitem{BaBo}
{\sc {Barthe}, F., and {Bordenave}, C.}
\newblock {Combinatorial optimization over two random point sets.}
\newblock In {\em {S\'eminaire de probabilit\'es XLV}}. Cham: Springer, 2013,
  pp.~483--535.

\bibitem{beardwood1959shortest}
{\sc Beardwood, J., Halton, J.~H., and Hammersley, J.~M.}
\newblock The shortest path through many points.
\newblock In {\em Mathematical Proceedings of the Cambridge Philosophical
  Society\/} (1959), vol.~55, Cambridge University Press, pp.~299--327.

\bibitem{BeCa}
{\sc Benedetto, D., and Caglioti, E.}
\newblock Euclidean random matching in 2d for non constant densities.
\newblock {\em arXiv preprint arXiv:1911.10523\/} (2019).

\bibitem{benedetto2021random}
{\sc Benedetto, D., Caglioti, E., Caracciolo, S., d’Achille, M., Sicuro, G.,
  and Sportiello, A.}
\newblock Random assignment problems on 2d manifolds.
\newblock {\em Journal of Statistical Physics 183}, 2 (2021), 1--40.

\bibitem{caffarelli1996boundary}
{\sc Caffarelli, L.~A.}
\newblock Boundary regularity of maps with convex potentials--{II}.
\newblock {\em Annals of mathematics 144}, 3 (1996), 453--496.

\bibitem{capelli2018exact}
{\sc Capelli, R., Caracciolo, S., Di~Gioacchino, A., and Malatesta, E.~M.}
\newblock Exact value for the average optimal cost of the bipartite traveling
  salesman and two-factor problems in two dimensions.
\newblock {\em Physical Review E 98}, 3 (2018), 030101.

\bibitem{CaLuPaSi14}
{\sc Caracciolo, S., Lucibello, C., Parisi, G., and Sicuro, G.}
\newblock Scaling hypothesis for the {E}uclidean bipartite matching problem.
\newblock {\em Physical Review E 90}, 1 (2014).

\bibitem{CaSi15}
{\sc Caracciolo, S., and Sicuro, G.}
\newblock Scaling hypothesis for the euclidean bipartite matching problem. ii.
  correlation functions.
\newblock {\em Physical Review E 91}, 6 (2015).

\bibitem{DeScSc13}
{\sc Dereich, S., Scheutzow, M., and Schottstedt, R.}
\newblock Constructive quantization: approximation by empirical measures.
\newblock {\em Ann. Inst. Henri Poincar\'e Probab. Stat. 49}, 4 (2013),
  1183--1203.

\bibitem{FigGig}
{\sc Figalli, A., and Gigli, N.}
\newblock A new transportation distance between non-negative measures, with
  applications to gradients flows with {D}irichlet boundary conditions.
\newblock {\em Journal de math{\'e}matiques pures et appliqu{\'e}es 94}, 2
  (2010), 107--130.

\bibitem{goldman2021fluctuation}
{\sc Goldman, M., and Huesmann, M.}
\newblock A fluctuation result for the displacement in the optimal matching
  problem.
\newblock {\em arXiv preprint arXiv:2105.02915\/} (2021).

\bibitem{GHO}
{\sc {Goldman}, M., {Huesmann}, M., and {Otto}, F.}
\newblock Quantitative linearization results for the {M}onge-{A}mp\`ere
  equation.

\bibitem{GolTre}
{\sc Goldman, M., and Trevisan, D.}
\newblock Convergence of asymptotic costs for random {E}uclidean matching
  problems.
\newblock {\em arXiv preprint arXiv:2009.04128\/} (2020).

\bibitem{hajlasz2000sobolev}
{\sc Hajlasz, P., and Koskela, P.}
\newblock Sobolev met {P}oincar\'{e}.
\newblock {\em Mem. Amer. Math. Soc. 145}, 688 (2000).

\bibitem{jost2013compact}
{\sc Jost, J.}
\newblock {\em Compact Riemann surfaces: an introduction to contemporary
  mathematics}.
\newblock Springer Science \& Business Media, 2013.

\bibitem{Le17}
{\sc Ledoux, M.}
\newblock On optimal matching of {G}aussian samples.
\newblock {\em Zap. Nauchn. Sem. S.-Peterburg. Otdel. Mat. Inst. Steklov.
  (POMI) 457}, Veroyatnost' \ i Statistika. 25 (2017), 226--264.

\bibitem{Le18}
{\sc Ledoux, M.}
\newblock On optimal matching of {G}aussian samples {II}.
\newblock {\em preprint\/} (2018).

\bibitem{leighton1989tight}
{\sc Leighton, T., and Shor, P.}
\newblock Tight bounds for minimax grid matching with applications to the
  average case analysis of algorithms.
\newblock {\em Combinatorica 9}, 2 (1989), 161--187.

\bibitem{lovasz2009matching}
{\sc Lov{\'a}sz, L., and Plummer, M.~D.}
\newblock {\em Matching theory}, vol.~367.
\newblock American Mathematical Soc., 2009.

\bibitem{41870}
{\sc Mehta, A.}
\newblock Online matching and ad allocation.
\newblock {\em Foundations and Trends in Theoretical Computer Science 8 (4)\/}
  (2013), 265--368.

\bibitem{mezard2009information}
{\sc Mezard, M., and Montanari, A.}
\newblock {\em Information, physics, and computation}.
\newblock Oxford University Press, 2009.

\bibitem{mezard1987spin}
{\sc M{\'e}zard, M., Parisi, G., and Virasoro, M.~A.}
\newblock {\em Spin glass theory and beyond: An Introduction to the Replica
  Method and Its Applications}, vol.~9.
\newblock World Scientific Publishing Company, 1987.

\bibitem{peyre2019computational}
{\sc Peyr{\'e}, G., Cuturi, M., et~al.}
\newblock Computational optimal transport: With applications to data science.
\newblock {\em Foundations and Trends{\textregistered} in Machine Learning 11},
  5-6 (2019), 355--607.

\bibitem{peyre2018comparison}
{\sc Peyre, R.}
\newblock {Comparison between $\mathrm{W}_2$ distance and
  $\dot{\mathrm{H}}^{-1}$ norm, and Localization of Wasserstein distance}.
\newblock {\em {ESAIM: Control, Optimisation and Calculus of Variations} 24}, 4
  (2018), 1489--1501.

\bibitem{redmond1996asymptotics}
{\sc Redmond, C., and Yukich, J.}
\newblock Asymptotics for {E}uclidean functionals with power-weighted edges.
\newblock {\em Stochastic processes and their applications 61}, 2 (1996),
  289--304.

\bibitem{rosenthal1970subspaces}
{\sc Rosenthal, H.~P.}
\newblock On the subspaces of ${L}^p$ ($p> 2$) spanned by sequences of
  independent random variables.
\newblock {\em Israel Journal of Mathematics 8}, 3 (1970), 273--303.

\bibitem{Santam}
{\sc Santambrogio, F.}
\newblock {\em Optimal transport for applied mathematicians}, vol.~87 of {\em
  Progress in Nonlinear Differential Equations and their Applications}.
\newblock Birkh\"auser/Springer, Cham, 2015.
\newblock Calculus of variations, PDEs, and modeling.

\bibitem{sicuro_euclidean_2017}
{\sc Sicuro, G.}
\newblock {\em The {Euclidean} {Matching} {Problem}}.
\newblock Springer {Theses}. Springer International Publishing, 2017.

\bibitem{steele1997probability}
{\sc Steele, J.~M.}
\newblock {\em Probability theory and combinatorial optimization}.
\newblock SIAM, 1997.

\bibitem{stein2016singular}
{\sc Stein, E.~M.}
\newblock {\em Singular integrals and differentiability properties of functions
  (PMS-30), volume 30}.
\newblock Princeton university press, 2016.

\bibitem{talagrand1992ajtai}
{\sc Talagrand, M.}
\newblock The {A}jtai-{K}oml{\'o}s-{T}usn{\'a}dy matching theorem for general
  measures.
\newblock In {\em Probability in Banach Spaces, 8: Proceedings of the Eighth
  International Conference\/} (1992), Springer, pp.~39--54.

\bibitem{Ta92}
{\sc Talagrand, M.}
\newblock Matching random samples in many dimensions.
\newblock {\em Ann. Appl. Probab. 2}, 4 (11 1992), 846--856.

\bibitem{Ta14}
{\sc Talagrand, M.}
\newblock {\em Upper and lower bounds for stochastic processes: modern methods
  and classical problems}, vol.~60.
\newblock Springer Science \& Business Media, 2014.

\bibitem{talagrand2018scaling}
{\sc Talagrand, M.}
\newblock Scaling and non-standard matching theorems.
\newblock {\em Comptes Rendus Mathematique 356}, 6 (2018), 692--695.

\bibitem{trillos2015rate}
{\sc Trillos, N.~G., and Slep{\v{c}}ev, D.}
\newblock On the rate of convergence of empirical measures in
  {$\infty$}-transportation distance.
\newblock {\em Canadian Journal of Mathematics 67}, 6 (2015), 1358--1383.

\bibitem{varopoulos2008analysis}
{\sc Varopoulos, N.~T., Saloff-Coste, L., and Coulhon, T.}
\newblock {\em Analysis and geometry on groups}.
\newblock No.~100. Cambridge university press, 2008.

\bibitem{Viltop}
{\sc Villani, C.}
\newblock {\em Topics in optimal transportation}, vol.~58 of {\em Graduate
  Studies in Mathematics}.
\newblock American Mathematical Society, Providence, RI, 2003.

\bibitem{yukich1995asymptotics}
{\sc Yukich, J.}
\newblock Asymptotics for the {E}uclidean {TSP} with power weighted edges.
\newblock {\em Probability theory and related fields 102}, 2 (1995), 203--220.

\bibitem{yukich2000asymptotics}
{\sc Yukich, J.}
\newblock Asymptotics for weighted minimal spanning trees on random points.
\newblock {\em Stochastic Processes and their applications 85}, 1 (2000),
  123--138.

\bibitem{Yu98}
{\sc {Yukich}, J.~E.}
\newblock {\em {Probability theory of classical {E}uclidean optimization
  problems.}}
\newblock Berlin: Springer, 1998.

\end{thebibliography}
 \end{document}